\PassOptionsToPackage{unicode=true}{hyperref} 
\PassOptionsToPackage{hyphens}{url}
\documentclass[11pt]{amsart}
\usepackage{lmodern}
\usepackage{amssymb,amsmath,amsthm}
\newtheorem{lemma}{Lemma}
\newtheorem{prop}{Proposition}
\newtheorem{cor}{Corollary}
\newtheorem*{main}{Main Theorem}
\usepackage{ifxetex,ifluatex}
\usepackage{fixltx2e} 
\ifnum 0\ifxetex 1\fi\ifluatex 1\fi=0 
  \usepackage[T1]{fontenc}
  \usepackage[utf8]{inputenc}
  \usepackage{textcomp} 
\else 
  \usepackage{unicode-math}
  \defaultfontfeatures{Ligatures=TeX,Scale=MatchLowercase}
\fi
\IfFileExists{upquote.sty}{\usepackage{upquote}}{}
\IfFileExists{microtype.sty}{%
\usepackage[]{microtype}
\UseMicrotypeSet[protrusion]{basicmath} 
}{}
\IfFileExists{parskip.sty}{%
\usepackage{parskip}
}{
\setlength{\parindent}{0pt}
\setlength{\parskip}{6pt plus 2pt minus 1pt}
}
\numberwithin{equation}{section}
\usepackage{hyperref}
\hypersetup{
            pdftitle={Expansion, divisibility and parity: an explanation},
            pdfborder={0 0 0},
            breaklinks=true}
\usepackage{longtable,booktabs}
\IfFileExists{footnote.sty}{\usepackage{footnote}\makesavenoteenv{longtable}}{}
\providecommand{\tightlist}{%
  \setlength{\itemsep}{0pt}\setlength{\parskip}{0pt}}
\ifx\paragraph\undefined\else
\let\oldparagraph\paragraph
\renewcommand{\paragraph}[1]{\oldparagraph{#1}\mbox{}}
\fi
\ifx\subparagraph\undefined\else
\let\oldsubparagraph\subparagraph
\renewcommand{\subparagraph}[1]{\oldsubparagraph{#1}\mbox{}}
\fi

\makeatletter
\def\fps@figure{htbp}
\makeatother

\usepackage{tikz,pgfplots,mathrsfs}
\usetikzlibrary{graphs, graphs.standard,math,positioning}

\title[Expansion, divisibility and parity: an explanation]{Expansion, divisibility and parity:\\ an explanation}
\author{Harald Andr\'es Helfgott}
\date{}
\address{Harald A. Helfgott, IMJ-PRG, UMR 7586,
 58 avenue de France, B\^{a}timent S. Germain, case 7012,
 75013 Paris CEDEX 13, France;
 Mathematisches Institut,
 Georg-August Universit\"{a}t G\"{o}ttingen, Bunsenstra{\ss}e 3-5,
 D-37073 G\"{o}ttingen, Deutschland}
\email{harald.helfgott@gmail.com}

\begin{document}
\begin{abstract}
After seeing how questions on the finer distribution of prime
factorization -- considered inaccessible until recently -- reduce to bounding
the norm of an operator defined on a graph describing factorization, we will
show how to bound that norm. In essence, the graph is a strong local expander,
with all eigenvalues bounded by a constant factor times the theoretical minimum
(i.e., the eigenvalue bound corresponding to Ramanujan graphs).
The proof will take us on a walk from graph
theory to linear algebra and the geometry of numbers, and back to graph theory,
aided, along the way, by a generalized sieve. This is an expository paper;
the full proof has appeared as a joint preprint with M. Radziwiłł.
\end{abstract}

\maketitle

\section{Introduction}
This paper is meant as an informal exposition of \cite{HelfRadz}. The main result is a statement on how a linear
operator defined in terms of divisibility by primes has small norm. In
this exposition, we will choose to start from one of its main current
applications, namely,

\begin{equation}\label{eq:firstdagger}
\frac{1}{\log x} \sum_{n\leq x} \frac{\lambda(n) \lambda(n+1)}{n} = O\left(
\frac{1}{\sqrt{\log \log x}}\right),
\end{equation} which strengthens results by Tao \cite{MR3569059} and
Tao-Teräväinen \cite{zbMATH07141311}. (Here
\(\lambda(n)\) is the Liouville function, viz., the completely
multiplicative function such that \(\lambda(p)=-1\) for every prime
\(p\).) There are other corollaries, some of them subsuming the above
statement. It is also true that above statement is an improvement on a
bound, whereas the main result is a result that is new also in a
qualitative sense. One may thus ask oneself whether it is right to
center the exposition on \eqref{eq:firstdagger}.

All the same, \eqref{eq:firstdagger} is a concrete statement that is obviously
interesting, being a step towards Chowla's conjecture (``logarithmic
Chowla in degree 2''), and so it is a convenient initial goal.

\begin{center}\rule{0.5\linewidth}{0.5000000000pt}\end{center}

First, some meta comments. We may contrast two possible ways of writing
a paper --- what may be called the \emph{incremental} and the
\emph{retrospective} approaches.

\begin{itemize}
\tightlist
\item
  In the \emph{incremental} approach, we write a paper while we solve a
  problem, letting complications and detours accrete. There is much that
  can be said against this approach: it is hard to distinguish it from
  simply lazy writing; the end product may be unclear; just how one got
  to the solution of the problem may be inessential or even misleading.
\item
  The \emph{retrospective} approach consists in writing the paper once
  the proof is done, from the perspective that one has reached by the
  time one has solved the problem.
\end{itemize}

These are of course two extremes. Few people write nothing down while
solving a problem, and the way one followed to reach the solution
generally has some influence on the finished paper. In the case of my
paper with Maksym, what we followed was mainly the retrospective
approach, with some incremental elements, mainly to deal with technical
complications that we had to deal with after we had an outline of a
proof. It is tempting to say that that is still too much incrementality,
but, in fact, some of the feedback we have received suggests a drawback
of the retrospective approach that I had not thought of before.

When faced with a result with a lengthy proof, readers tend to come up
with their own ``natural strategy''. So far, so good: active reading is
surely a good thing. What then happens, though, is that readers may see
necessary divergences from their ``natural strategy'' as technical
complications. They may often be correct; however, they may miss why the
``natural strategy'' may not work, or how it leads to the main,
essential difficulty --- the heart of the problem, which they may then
miss for following the complications.

What I will do in this write-up is follow, not an incremental approach,
but rather an idealized view of what the path towards the solution was
or could have been like; a recreated incrementality with the benefit of
hindsight, then, starting from a ``natural strategy'', with an emphasis
on what turns out to be essential.

{\em Notation.} We will use notation that is usual within
analytic number theory. In particular, given two functions $f,g$ on
$\mathbb{R}^+$ or $\mathbb{Z}^+$,
     $f(x) = O(g(x))$ means that there exists a constant $C>0$
  such that $|f(x)|\leq C g(x)$ for all large enough $x$, and
  $f(x) = o(g(x))$ means that $\lim_{x\to \infty} f(x)/g(x) = 0$ (and $g(x)>0$ for $x$ large enough).
By \(O^*(B)\), we will mean ``a quantity whose absolute value is no
larger than \(B\)''; it is a useful bit of notation for error terms.
We define \(\omega(n)\) to be the number of prime divisors of an
integer \(n\).

\subsection{Initial setup}

Let us set out to reprove Tao's ``logarithmic Chowla'' statement, that
is,

\[
\frac{1}{\log x} \sum_{n\leq x} \frac{\lambda(n) \lambda(n+1)}{n} \to 0
\]

as \(x\to \infty\). Now, Tao's method gives a bound of
\(O(1/(\log \log \log \log x)^{\alpha})\) on the left side (as explained
in \cite{HelfUbis}, with \(\alpha=1/5\)),
while Tao-Teräväinen should yield a bound of
\(O(1/(\log \log \log x)^{\alpha})\) for some \(\alpha>0\). Their work
is based on depleting entropy, or, more precisely, on depleting mutual
information. Our method gives stronger bounds (namely,
\(O(1/\sqrt{\log \log x})\)) and is also ``stronger'' in ways that will
later become apparent. Let us focus, however, simply on giving a
different proof, and welcome whatever might come from it.

The first step will be consist of a little manipulation as in Tao, based
on the fact that \(\lambda\) is multiplicative. Let
\(W = \sum_{n\leq x} \lambda(n) \lambda(n+1)/n\). For any prime (or
integer!) \(p\),

\[
\begin{aligned}
\frac{1}{p} W &= \sum_{n\leq x} \frac{\lambda(p n) \lambda(p n + p)}{p n}\\
&= \sum_{n\leq p x:\; p|n} \frac{\lambda(n) \lambda(n+p)}{n} = 
\sum_{n\leq x:\; p|n} \frac{\lambda(n) \lambda(n+p)}{n}
+ O\left(\frac{\log p}{p}\right).\end{aligned}
\]

Hence, for any set of primes \(\mathbf{P}\),

\[
\sum_{p\in \mathbf{P}} \sum_{n\leq x:\; p|n} \frac{\lambda(n) \lambda(n+p)}{n} = W \mathscr{L} + O\left(\sum_{p\in \mathbf{P}} 
\frac{\log p}{p}\right),
\]

where \(\mathscr{L} = \sum_{p\in \mathbf{P}} 1/p\). If \(H\) is such
that \(p\leq H\) for all \(p\in \mathbf{P}\), then, by the prime number
theorem, \(\sum_{p\in \mathbf{P}} (\log p)/p \ll \log H\). Thus

\[
W = \frac{1}{\mathscr{L}} \sum_{n\leq x} \sum_{p\in \mathbf{P}: p|n}
\frac{\lambda(n) \lambda(n+p)}{n} + O\left(\frac{\log H}{\mathscr{L}}\right).
\]

Assuming \(H = x^{o(1)}\) (so that \(\log H = o(\log x)\)) and
\(\mathscr{L}\geq 1\), and using a little partial summation, we see that, to
prove that $W = o(\log x)$,
it is enough to show that \(S_0 = o(N\mathscr{L})\), where

\[
S_0 = \sum_{N < n\leq 2 N}\sum_{p\in \mathbf{P}: p|n} \lambda(n) \lambda(n+p).
\]

Let us make this sum a little more symmetric. Let
\(\mathbf{N} = \{n\in \mathbb{Z}: N < n\leq 2 N\}\), and define

\[
S = \sum_{n\in \mathbf{N}} \sum_{\sigma = \pm 1} \sum_{\substack{p\in \mathbf{P}: p|n\\ n+\sigma p\in \mathbf{N}}} 
 \lambda(n) \lambda(n+\sigma p).
\]

Then \(S = 2 S_0 + O(\sum_{p\in \mathbf{P}} 1) = 2 S_0 + O(H)\), and
thus it is enough to prove that

\[
S = o(N \mathscr{L}).
\]

\textbf{\emph{Objectives.}} Tao showed that there \emph{exists} a set
\(\mathbf{P}\) of primes (very small compared to \(N\)) such that
\(S = o(N \mathscr{L})\). It is our aim to prove that
\(S = o(N \mathscr{L})\) for \emph{every} set \(\mathbf{P}\) of primes
satisfying some simple conditions. (As we said, we assume \(p\leq H\),
and it is not hard to see that we have to assume
\(\mathscr{L}\to\infty\); it will also be helpful to assume that no
\(p\in \mathbf{P}\) is tiny compared to \(H\).) We will in fact be able
to show that \(S = O(N \sqrt{\mathscr{L}})\), which is essentially
optimal.

\section{A first attempt}

We now set out on our own.

\subsection{Old habits die hard. A reduction}

It is a deep-seated instinct for an analytic
number theorist to apply Cauchy-Schwarz:

\[
\begin{aligned}
S^2 &\leq \left(\sum_{n\in \mathbf{N}} \sum_{\sigma = \pm 1} \sum_{\substack{p\in \mathbf{P}: p|n\\ n+\sigma p\in \mathbf{N}}} 
 \lambda(n) \lambda(n+\sigma p)\right)^2\\
 &\leq N \sum_{n\in \mathbf{N}} \sum_{\sigma_1,\sigma_2 = \pm 1} \sum_{\substack{p_1,p_2\in \mathbf{P}\\ p_i|n,\; n+\sigma_i p_i\in \mathbf{N}}} 
  \lambda(n+\sigma_1 p_1) \lambda(n+\sigma_2 p_2)\\
  &\leq \sum_{n\in \mathbf{N}}\; \sum_{\sigma_1,\sigma_2 = \pm 1} \sum_{\substack{p_1,p_2\in \mathbf{P}\\ p_1|n, p_2|n+\sigma_1 p_1\\
  n+\sigma_1 p_1, n+\sigma_1 p_1 + \sigma_2 p_2 \in \mathbf{N}}} \lambda(n)
\lambda(n+\sigma_1 p_1+\sigma_2 p_2),\end{aligned}
\]

where we are changing variables in the last step.

We iterate, applying Cauchy-Schwarz \(\ell\) times:

\[
S^{2^\ell} \leq N^{2^\ell-1} \sum_{n\in \mathbf{N}}
\sum_{\substack{\sigma_i= \pm 1,\; p_i\in \mathbf{P}\\ \forall 1\leq i\leq 2^\ell: p_i|n+\sigma_1 p_1 + \dotsc + \sigma_{i-1} p_{i-1}}} \lambda(n) 
\lambda(n+ \sigma_1 p_1 + \dotsc + \sigma_{2^\ell} p_{2^{\ell}}).
\]

We can see \(n+\sigma_1 p_1 + \dotsc + \sigma_{2^\ell} p_{2^\ell}\) as
the outcome of a ``walk'' of length \(2^\ell\).

Suppose for a moment that, for \(k=2^\ell\) large, the number of walks
of length \(k\) from \(n\) to \(m\) is generally about \(\psi(m)\),
where \(\psi\) is a nice continuous function. Then \(S^{2^\ell}\) would
tend to

\[
N^{2^\ell-1} \sum_{n\in \mathbf{N}} \sum_{m} \lambda(n) \lambda(n+m) \psi(m).
\]

The main result (Theorem 1) in Matomäki-Radziwiłł \cite{MR3488742} would then
give us a bound on that double sum. Let us write that bound in the form

\[
\left|\sum_{n\in \mathbf{N}} \sum_{m} \lambda(n) \lambda(n+m) \psi(m)\right|
\leq \text{err}_2 \cdot N \mathscr{L}^{k},
\]

since \(|\psi|_1\) should be about \(\mathscr{L}^k\), and write our
statement on convergence to \(\psi\) in the form

\begin{equation}\label{eq:diffabs}
\sum_{n\in \mathbf{N}} \sum_m \left|\psi(m) - 
\sum_{\substack{\sigma_i= \pm 1,\; p_i\in \mathbf{P}\\ \forall 1\leq i\leq 2^\ell: p_i|n+\sigma_1 p_1 + \dotsc + \sigma_{i-1} p_{i-1}\\ \sigma_1 p_1 + \dotsc + 
\sigma_k p_k = m}} 1\right| = \text{err}_1 \cdot N \mathscr{L}^k.
\end{equation}

Then

\[
S\leq (\text{err}_1^{1/k} + \text{err}_2^{1/k})\cdot N \mathscr{L}.
\]

Here already we would seem to have a problem. The ``width'' \(M\) of the
distribution \(\psi\) (meaning its scale) should be
\(\ll \sqrt{k} \cdot \mathbb{E}(p: p\in \mathbb{P})\leq \sqrt{k} H\);
the distribution could be something like a Gaussian at that scale, say.
Now, the bound from \cite{MR3488742} is roughly
of the quality \(\text{err}_2\leq 1/\log M\). One can use intermediate
results in the same paper to obtain a bound on \(\text{err}_2\) roughly
of the form \(1/M^{\delta}\), \(\delta>0\), if we remove some integers
from \(\mathbf{N}\). At any rate, it seems clear that we would need, at the
very least, \(k\) larger than any constant times \(\log H\).

As it turns out, all of that is a non-issue, in that there is a way to avoid
taking the \(k\)th root of \(\text{err}_2\) altogether. Let us make a mental
note, however.

\subsection{Walks of different kinds}

The question now is how large \(\ell\) has to be for the number of walks
of length \(k=2^{\ell}\) from \(n\) to \(n+m\) to approach a continuous
distribution \(\psi(m)\). Consider first the walks
\(n, n+\sigma_1 p_1,\dotsc,n+\sigma_1 p_1 + \dotsb + \sigma_k p_k\) such
that no prime \(p_i\) is repeated. Fix \(\sigma_i\), \(p_i\) and let
\(n\) vary. By the Chinese Remainder Theorem, the number of
\(n\in \mathbf{N}\) such that

\[
p_1|n,\; p_2|n+\sigma_1 p_1,\; \dotsc,\; p_k|n+\sigma_1 p_1 + \dotsc + \sigma_{k-1} p_{k-1}
\]

is almost exactly \(N/p_1 p_2 \dotsb p_k\). In other words, the
probability of that walk being allowed is almost exactly
\(1/p_1 \dotsc p_k\). We may thus guess that \(\psi\) has the same shape
(scaled up by a factor of \(\mathscr{L}^k\)) as the distribution
of the endpoint of a
random walk where each edge of length \(p\) is taken with probability
\(1/p_i\) (divided by \(\mathscr{L}\), so that the probabilities add up
to \(1\)). That distribution should indeed tend to a continuous
distribution --- namely, a Gaussian --- fairly quickly. Of course, here,
we are just talking about the contribution of walks with distinct edges
\(p_i\) to

\[
\sum_{n\in \mathbf{N}} \sum_m \left(\psi(m) - 
\sum_{\substack{\sigma_i= \pm 1,\; p_i\in \mathbf{P}\\ \forall 1\leq i\leq 2^\ell: p_i|n+\sigma_1 p_1 + \dotsc + \sigma_{i-1} p_{i-1}\\ \sigma_1 p_1 + \dotsc + 
\sigma_k p_k = m}} 1\right),
\]

without absolute values, and we do need to take absolute values as in
\eqref{eq:diffabs}. However, we can get essentially what we want by
looking at the variance

\[
\sum_{n\in \mathbf{N}} \sum_m \left(\psi(m) - 
\sum_{\substack{\sigma_i= \pm 1,\; p_i\in \mathbf{P}\\ \forall 1\leq i\leq k: p_i|n+\sigma_1 p_1 + \dotsc + \sigma_{i-1} p_{i-1}\\ \sigma_1 p_1 + \dotsc + 
\sigma_k p_k = m}} 1\right)^2,
\]

and considering the contribution to this variance made by closed walks

\[
\begin{aligned}n, &n+\sigma_1 p_1, \dotsc, n+\sigma_1 p_1 + \dotsb + \sigma_k p_k = m,\\ 
&n+\sigma_1 p_1 + \dotsb + \sigma_k p_k-\sigma_{k+1} p_{k+1},\dotsc ,m-(\sigma_{k+1} p_{k+1} + \dotsc + \sigma_{2 k}
p_{2 k})=n\end{aligned}
\]

with \(p_1,p_2,\dotsc,p_{2 k}\) distinct:

\begin{center}
\begin{tikzpicture}[thick,scale=0.9, every node/.style={transform shape}]
\tikzstyle{membre}= [rectangle]
\tikzstyle{operation}=[->,>=latex]
\tikzstyle{etiquette}=[midway,fill=black!20]
\node[membre] (n) at (0,3) {$n$};
\node[membre] (n1) at (2.5,4) {$n+\sigma_1 p_1$};
\node[membre] (n2) at (6,4.5) {$n+\sigma_1 p_1 + \sigma_2 p_2$};
\node[membre] (aro) at (9.5,4) {$\dots$};
\node[membre] (m) at (12,3) {$n+\sigma_1 p_1 + \dots + \sigma_k p_k$};
\node[membre] (nk1) at (8.5,1) {$n+\sigma_1 p_1 + \dots + \sigma_k p_k - \sigma_{k+1} p_{k+1}$};
\node[membre] (nk2m1) at (2.75,1) {$\dots$};
\draw[operation] (n) to node[midway,above]{$\sigma_1 p_1$} (n1); 
\draw[operation] (n1) to node[midway,above] {$\sigma_2 p_2$} (n2);
\draw[operation] (n2) to node[midway,above] {$\sigma_3 p_3$} (aro);
\draw[operation] (aro) to node[midway,above] {$\;\;\sigma_k p_k$} (m);
\draw[operation] (m) to node[midway,below] {$\;\;\;\;\;\;\;\;\;\;\;\;\sigma_{k+1} p_{k+1}$} (nk1);
\draw[operation] (nk1) to node[midway,above] {$\sigma_{k+2} p_{k+2}$} (nk2m1);
\draw[operation] (nk2m1) to node[midway,right] {$\sigma_{2 k} p_{2 k}$} (n);
\end{tikzpicture}
\end{center}

The contribution of these closed walks is almost exactly what we would
obtain from the naïve model we were implicitly considering, viz., a random walk
where each edge \(p_i\) is taken with probability
\(1/(\mathscr{L} p_i)\), and so we should have the same limiting
distribution as in that model.

What about walks where some primes \(p_i\) do repeat? At least some of
them may make a large contribution that is not there in our naïve model.
For instance, consider walks of length \(2 k\) that retrace their steps,
so that the \((n+1)\)th step is the \(n\)th step backwards, the
\((n+2)\)th step is the \((n-1)\)th step backwards, etc.:

\[
\begin{aligned}
n, &n+\sigma_1 p_1, \dotsc, n+\sigma_1 p_1 + \dotsb + \sigma_k p_k,\\ 
&n+\sigma_1 p_1 + \dotsb + \sigma_{k-1} p_{k-1},\dotsc ,n+\sigma_1 p_1, n,\end{aligned}
\]

with

\[
p_1|n,\; p_2|n+\sigma_1 p_1,\; \dotsc,\; p_k|n+\sigma_1 p_1 + \dotsc + \sigma_{k-1} p_{k-1},
\]

\[
p_k|n+\sigma_1 p_1 + \dotsc + \sigma_{k-1} p_{k-1} + \sigma_k p_k,\;
\dotsc,\; p_2|n+\sigma_1 p_1 + \sigma_2 p_2,\; p_1|n+\sigma_1 p_1.
\]

The second row of divisibility conditions here is obviously implied by
the first row. Hence, again by the Chinese Remainder Theorem, the walk
is valid for almost exactly \(N/p_1 p_2 \dotsb p_k\) elements
\(n\in \mathbf{N}\), rather than for \(N/(p_1 p_2 \dotsb p_k)^2\)
elements. The contribution of such walks to

\[
\sum_{n\in \mathbf{N}}
\sum_{\substack{\forall 1\leq i\leq 2 k: \sigma_i= \pm 1,\; p_i\in \mathbf{P}\\ \forall 1\leq i\leq 2 k: p_i|n+\sigma_1 p_1 + \dotsc + \sigma_{i-1} p_{i-1}\\ \sigma_1 p_1 + \dotsc + 
\sigma_{2 k} p_{2 k} = 0}} 1
\]

(which is the interesting part of the variance we wrote down before) is
clearly \(N \mathscr{L}^k\). In order for it not to be of greater order
than what one expects from the limiting distribution, we should have
\(N \mathscr{L}^k \ll N \mathscr{L}^{2 k}/M\), where \(M\), the width of
the distribution, is, as we saw before, very roughly \(\sqrt{k} H\).
Thus, we need \(k\gg (\log H)/(\log \mathscr{L})\).

There are of course other walks that make similar contributions; take,
for instance,

\[
n, n+p_1, n, n-p_3, n-p_3 + p_4, n-p_3, n-p_3+p_6, n-p_3, n
\]

for \(k=3\). These are what we may call \emph{trivial walks}, in the
sense that a word is \emph{trivial} when it reduces to the identity. It
is tempting to say that their number is \(2^k C_k\), where
\(C_k\leq 2^{2 k}\) is the \(k\)th Catalan number (which, among other
things, counts the number of expressions containing \(k\) pairs of
parentheses correctly matched: for example, \(() (())\) would correspond
to the trivial walk above). In fact, the matter becomes more subtle
because some primes may reappear without taking us one step further back
to the origin of the walk; for instance, in the above, we might have
\(p_4=p_1\), and that is a possibility that is not recorded by a simple
pattern of correctly matched parentheses--- yet it must be considered
separately. Here again we make a mental note. 

 It is, incidentally, no coincidence that, when we try to draw the
trivial walk above, we produce a tree:

\begin{center}
\begin{tikzpicture}[thick,scale=0.9, every node/.style={transform shape}]
\tikzstyle{membre}= [rectangle]
\tikzstyle{operation}=[->,>=latex]
\tikzstyle{etiquette}=[midway,fill=black!20]
\node[membre] (n) at (0,2) {$n$};
\node[membre] (n1) at (3,3) {$n+p_1$};
\node[membre] (n3) at (3,1) {$n-p_3$};
\node[membre] (n34) at (6,2) {$n-p_3+p_4$};
\node[membre] (n36) at (6,0) {$n-p_3+p_6$};
\draw[operation] (n)--(n1);
\draw[operation] (n)--(n3);
\draw[operation] (n3)--(n34);
\draw[operation] (n3)--(n36);
\end{tikzpicture}
\end{center}

Any trivial walk gives us a tree (or rather a tree traversal) when drawn.

Now let us look at walks that fall into neither of the two classes just
discussed; that is, walks where we do have some repeated primes
\(p_i=p_{i'}\) even after we reduce the walk.
(When we say we
\emph{reduce} a walk, we mean an analogous procedure to that of reducing
a word.)
Then, far from being
independent, the condition

\[
p_i|n + \sigma_1 p_1 + \dotsc + \sigma_{i-1} p_{i-1}
\]

either implies or contradicts the condition

\[
p_i=p_{i'}|n+\sigma_1 p_1 + \dotsc + \sigma_{i'-1} p_{i'-1}
\]

for given \(\{(\sigma_i,p_i)\}_i\), depending on whether

\[
p_i|\sigma_i p_i + \sigma_{i+1} p_{i+1} + \dotsc + \sigma_{i'-1} p_{i'-1}.
\]

We may draw another graph, emphasizing the two edges with the same label
\(\pm p_i\):

\begin{center}
\begin{tikzpicture}[thick,scale=0.9, every node/.style={transform shape}]
\tikzstyle{membre}= [rectangle]
\tikzstyle{operation}=[->,>=latex]
\tikzstyle{etiquette}=[midway,fill=black!20]
\node[membre] (nim1) at (0,0) {$n+\dots +\sigma_{i-1} p_{i-1}$};
\node[membre] (ni) at (1,2) {$n+\dots+\sigma_i p_i$};
\node[membre] (nmid) at (5,2.5) {$\dots$};
\node[membre] (nim) at (10,2) {$n+\dots+\sigma_i p_i + \dots + \sigma_{i'-1} p_{i'-1}$};
\node[membre] (nip) at (9,0) {$n+\dots+\sigma_i p_i + \dots +
  \sigma_{i'} p_{i'}$};
\draw[operation, ultra thick, densely dashed] (nim1)--(ni) node[midway,left]{$\sigma_i p_i$};
\draw[operation] (ni) to[out=20,in=160] (nmid);
\draw[operation] (nmid) to[out=20,in=160] (nim);
\draw[operation, ultra thick, densely dashed] (nim)--(nip) node[midway,left]{$\sigma_{i'} p_{i'}=\sigma_{i'} p_i$};
\end{tikzpicture}
\end{center}

At this point it becomes convenient to introduce the assumption that
\(p\geq H_0\) for all \(p\in \mathbf{P}\). Then it is clear that, if
\(i'-i>1\) and
\(p_j\ne p_i\) for all \(i < j < i'\), the divisibility condition
\(p_i|\sigma_{i+1} p_{i+1} + \dotsc + \sigma_{i'-1} p_{i'-1}\) may
hold only for a proportion \(\ll 1/H_0\) of all tuples
\((p_{i+1},\dotsc,p_{i'-1})\).

So far, so good, except that it
is not enough to save one factor of \(H_0\), and indeed we should save a
factor of at least \(M\), which is roughly in the scale of \(H\), not
\(H_0\). Obviously, for \(\mathscr{L}\to \infty\) to hold, we need
\(H_0 = H^{o(1)}\), and so we need to save more than any constant number
of factors of \(H_0\).

We have seen three rather different cases. In general, we would like to
have a division of all walks into three classes:

\begin{enumerate}
\def\labelenumi{\arabic{enumi}.}
\tightlist
\item
  walks containing enough non-repeated primes \(p_i\) that their
  contribution is one would expect from the hoped-for limiting
  distribution;
\item
  rare walks, such as, for example, trivial walks;
\item
  walks for which there are many independent conditions of the form
  \(p_i|n+\sigma_{i+1} p_{i+1} + \dotsc + \sigma_{i'-1} p_{i'-1}\) as
  above.
\end{enumerate}

\emph{\textbf{Some initial thoughts on the third case.}} We should think
a little about what we mean or should mean by ``independent''. It is
clear that, if we have several conditions \(p|L_j(p_1,\dotsc,p_{2 k})\),
where the \(L_j\) are linear forms spanning a space of dimension \(D\),
then, in effect, we have only \(D\) distinct conditions. It is also
clear that, while having several primes \(p_i\) divide the same quantity
\(L(p_1,\dotsc,p_{2 k})\) ought to give us more information than just
knowing one prime divides it, that is true only up to a point: if
\(L(p_1,\dotsc,p_{2 k})=0\) (something that we expect to happen about
\(1/\sqrt{k} H\) of the time), then every condition of the form
\(p_i|L(p_1,\dotsc,p_{2 k})\) holds trivially.

It is also the case that we should be careful about which primes do the
dividing. Say two indices \(i\), \(i'\) are equivalent if
\(p_i=p_{i'}\). Choose your equivalence relation \(\sim\), and paint the
indices \(i\) in some equivalence classes blue, while painting the
indices \(i\) in the other equivalence classes red. It is not hard to
show, using a little geometry of numbers, that, if
\(p_{i_j}|L_j(p_1,\dotsc,p_{2 k})\) for some blue indices \(i_j\) and
linear forms \(L_j\), \(j\in J\), and the space spanned by the forms
\(L_j\) \emph{considered as formal linear combinations on the variables}
\(x_i\) \emph{for} \(i\) \emph{red} is \(D\), we can gain a factor of
at least \(H_0^D\) or so: the primes \(p_i\) for \(i\) red have to lie
in a lattice of codimension \(D\) and index \(\geq H_0^D\). A
priori, however, it is not clear which primes we should color blue and
which ones red.

We have, at any rate, arrived at what may be called the core of the
problem -- how to classify our walks in three classes as above, and how
to estimate their contribution accordingly.

\section{Graphs, operators and eigenvalues}

It is now time to step back and take a fresh look at the problem.
Matters will become clearer and simpler, but, as we will see, the core
of the problem will remain.

We have been talking about walks. Now, walks are taken in a graph.
Thinking about it for a moment, we see that we have been considering
walks in the graph \(\Gamma\) having \(V=\mathbf{N}\) as its set of
vertices and
\(E=\{{n,n+p}: n,n+p\in \mathbf{N}, p\in \mathbf{P}, p|n\}\) as its set
of edges. (In other words, we draw an edge between \(n\) and \(n+p\) if
and only if \(p\) divides \(n\).) We also considered random walks in
what we called the ``naïve model''; those are walks in the weighted
graph \(\Gamma'\) having \(\mathbf{N}\) as its set of vertices and an
edge of weight \(1/p\) between any \(n, n+p\in \mathbf{N}\) with
\(p\in \mathbf{P}\), regardless of whether \(p|n\).

\subsection{Adjacency, eigenvalues and expansion}

Questions about walks in a graph \(\Gamma\) are closely tied to the
\emph{adjacency operator} \(\textrm{Ad}_\Gamma\). This is a linear
operator on functions \(f:V\to \mathbb{C}\) taking \(f\) to a function
\(\textrm{Ad}_\Gamma f:V\to \mathbb{C}\) defined as follows: for
\(v\in V\), \[
(\textrm{Ad}_\Gamma f)(v) = \sum_{w: \{v,w\}\in E} f(w).
\] In other words, \(\textrm{Ad}_\Gamma\) replaces the value of \(f\) at
a vertex \(v\) by the sum of its values \(f(w)\) at the neighbors \(w\)
of \(v\). The connection with walks is not hard to see: for instance, it
is very easy to show that, if \(1_v:V\to \mathbb{C}\) is the function
taking the value \(1\) at \(v\) and \(0\) elsewhere, then, for any
\(w\in V\) and any \(k\geq 0\), \(((\textrm{Ad}_\Gamma)^k 1_v)(w)\) is
the number of walks of length \(k\) from \(v\) to \(w\).

The connection between \(\textrm{Ad}_\Gamma\) and our problem is very
direct, in that it can be stated without reference to random walks. We
want to show that \[
\sum_{n\in \mathbf{N}} \sum_{\sigma = \pm 1} \sum_{\substack{p\in \mathbf{P}: p|n\\ n+\sigma p\in \mathbf{N}}} 
 \lambda(n) \lambda(n+\sigma p) = o(N \mathscr{L}).
\] That is exactly the same as showing that \[
\langle \lambda, \textrm{Ad}_{\Gamma} \lambda\rangle =
 o(\mathscr{L}),
\] where \(\langle \cdot, \cdot\rangle\) is the inner product defined by
\[
\langle f,g\rangle = \frac{1}{N} \sum_{n\in \mathbf{N}} f(n) \overline{g(n)}
\] for \(f,g:V\to \mathbb{C}\).

The behavior of random walks on a graph --- in particular, the limit
distribution of their endpoints --- is closely related to the notion of
\emph{expansion}. A regular graph \(\Gamma\) (that is, a graph where
every vertex has the same degree \(d\)) is said to be an \emph{expander
graph} with parameter \(\epsilon>0\) if, for every eigenvalue \(\gamma\)
of \(\textrm{Ad}_\Gamma\) corresponding to an eigenfunction orthogonal
to constant functions, \[|\gamma|\leq (1-\epsilon) d.\] {(A few basic
remarks may be in order. Since \(\Gamma\) is regular of degree \(d\), a
constant function on \(V\) is automatically an eigenfunction with
eigenvalue \(d\). Now, \(\textrm{Ad}_\Gamma\) is a symmetric operator,
and thus it has full real spectrum: the space of all functions
\(V\to \mathbb{C}\) is spanned by a set of eigenfunctions of
\(\textrm{Ad}_\Gamma\), all orthogonal to each other; the corresponding
eigenvalues are all real, and it is easy to see that all of them are at
most \(d\) in absolute value.)}

It is clear that we need something both stronger and weaker than
expansion. (We cannot use the definition of expansion
above ``as is'' anyhow, 
in that our graph \(\Gamma\) is not regular; its
\emph{average} degree is \(\mathscr{L}\).) We need a stronger bound than
what expansion provides: we want to show, not just that
\(|\langle \lambda, \textrm{Ad}_\Gamma \lambda\rangle|\) is
\(\leq (1-\epsilon) \mathscr{L}\), but that it is \(= o(\mathscr{L})\).
{ There is nothing unrealistically strong here --- in the strongest kind
of expander graph (\emph{Ramanujan graphs}), the absolute value of every
eigenvalue is at most \(2\sqrt{d-1}\).}

At the same time, we cannot ask for
\(\langle f, \textrm{Ad}_\Gamma f\rangle/|f|_2^2 = o(\mathscr{L})\) to
hold for every \(f\) orthogonal to constant functions. Take
\(f = 1_{I_1}-1_{I_2}\), where \(I_1\), \(I_2\) are two disjoint
intervals of the same length \(\geq 100 H\), say. Then \(f\) is
orthogonal to constant functions, but \((\textrm{Ad}_\Gamma f)(n)\) is
equal to \(\omega(n) f(n)\), except possibly for those \(n\) that lie at
a distance \(\leq H\) of the edges of \(I_1\) and \(I_2\). Hence,
\(\langle f,\textrm{Ad}_{\Gamma} f\rangle/|f|_2^2\) will be close to
\(\mathscr{L}\). It follows that \(\textrm{Ad}_\Gamma\) will have at
least one eigenfunction orthogonal to constant functions and with
eigenvalue close to \(\mathscr{L}\); in fact, it will have many.

{(This observation is related to the fact that endpoint of a short
random walk on \(\Gamma\) \emph{cannot} be approximately
equidistributed, as it is in an expander graph: the edges of \(\Gamma\)
are too short for that. The most we could hope for is what we were
aiming for, namely, that the distribution of the endpoint converges to a
nice distribution, centered at the starting point.)}

We could aim to show that
\(\langle f,\textrm{Ad}_\Gamma f\rangle/|f|_2^2\) is small whenever
\(f\) is approximately orthogonal to approximately locally constant
functions, say. Since the main result in \cite{MR3488742} can be
interpreted as the statement that \(\lambda\) is approximately
orthogonal to such functions, we would then obtain what we wanted to
prove for \(f=\lambda\).

We will find it cleaner to proceed slightly differently. Recall our
weighted graph \(\Gamma'\), which was meant as a naïve model for
\(\Gamma\). It has an adjacency operator \(\textrm{Ad}_{\Gamma'}\) as
well, defined as before. (Since \(\Gamma'\) has weights \(1/p\) on its
edges,
\((\textrm{Ad}_{\Gamma'} f)(n) = \sum_{p\in \mathbf{P}} (f(n+p) + f(n-p))/p\).)
It is not hard to show, using the techniques in \cite{MR3488742}, that
\[\langle \lambda,\textrm{Ad}_{\Gamma'} \lambda\rangle = o(\mathscr{L}).\]
(In fact, what amounts to this statement has already been shown, in
\cite[Lemma 3.4--3.5]{MR3569059}; the main
ingredient is \cite[Thm.~1.3]{MR3435814}, which
applies and generalizes the main theorem in \cite{MR3488742}. Their bound
is a fair deal smaller than \(o(\mathscr{L})\).) We define the operator
\[A = \textrm{Ad}_{\Gamma}-\textrm{Ad}_{\Gamma'}.\] It will then be
enough to show that
\[\langle \lambda,A\lambda\rangle = o(\mathscr{L}),\] as then it will
obviously follow that
\[\langle \lambda,\textrm{Ad}_\Gamma \lambda\rangle = 
\langle \lambda, A\lambda \rangle + \langle \lambda, \textrm{Ad}_{\Gamma'} \lambda \rangle = o(\mathscr{L}).\]

It would be natural to guess, and try to prove, that
\(\langle f, Af\rangle = o(\mathscr{L})\) for \emph{all}
\(f:V\to \mathbb{C}\) with \(|f|_2=1\), i.e., that all eigenvalues of
\(A\) are \(o(\mathscr{L})\).

We cannot hope for quite that much. The reason is simple. For any vertex
\(n\), \(\langle A 1_n, A 1_n\rangle\) equals the sum of the squares of
the weights of the edges \(\{n,n'\}\) containing \(n\). That sum equals
\[\sum_{\substack{p\in \mathbf{P}\\ p|n}} \left(1-\frac{1}{p}\right)^2 + 
\sum_{\substack{p\in \mathbf{P}\\ p\nmid n}} \frac{1}{p^2},\] which in
turn is greater than \(1/4\) times the number \(\omega_{\mathbf{P}}(n)\)
of divisors of \(n\) in \(\mathbf{P}\). Thus, \(A\) has at least one
eigenvalue greater than \(\sqrt{\omega_{\mathbf{P}}(n)}/2\). Now,
typically, \(n\) has about \(\mathscr{L}\) divisors in \(\mathbf{P}\),
but some integers \(n\) have many more; for some rare \(n\), in fact,
\(\omega_{\mathbf{P}}(n)\) will be greater than \(\mathscr{L}^2\), and
so there have to be eigenvalues of \(A\) greater than \(\mathscr{L}\).

It is thus clear that we will have to exclude some integers, i.e., we
will define our vertex set to be some subset
\(\mathscr{X}\subset \mathbf{N}\) with small complement. We will set
ourselves the goal of proving that all of the eigenvalues of the
operator \(A|_\mathscr{X}\) defined by
\[(A|_\mathscr{X})(f) = (A(f|_\mathscr{X}))|_\mathscr{X}\] are
\(o(\mathscr{L})\). (Here \(f|_\mathscr{X}\) is just the function taking
the value \(f(n)\) for \(n\in \mathscr{X}\) and \(0\) for
\(n\not\in \mathscr{X}\).) Then, for \(f=\lambda\), or for any other
\(f\) with \(|f|_\infty\leq 1\),
\[\langle f, A f\rangle = \langle f, (A|_\mathscr{X}) f\rangle +
O\left(\sum_{n\in \mathbf{N}\setminus \mathscr{X}} 2\, (\omega_{\mathbf{P}}(n)+
\mathscr{L})\right),\]
where, if \(\mathbf{N}\setminus \mathscr{X}\) is small enough (as it
will be), it will not be hard to show that the sum within \(O(\cdot)\)
is quite small. We will then be done: obviously
\(\langle f, (A|_\mathscr{X}) f\rangle\) is bounded by the largest
eigenvalue of \(A|_\mathscr{X}\) times \(|f|_2\) (which is
\(\leq |f|_\infty\leq 1\)), and so we will indeed have
\(\langle f, A f\rangle = o(\mathscr{L})\).

We will in fact be able to prove something stronger: there is a subset
\(\mathscr{X}\subset \mathbf{N}\) with small complement such that all
eigenvalues of \(A|_\mathscr{X}\) are \[O(\sqrt{\mathscr{L}}).\] (This
bound is optimal up to a constant factor.) This is our main theorem.

We hence obtain that
\begin{equation}\label{eq:feast}\langle \lambda, A\lambda\rangle = O(\sqrt{\mathscr{L}}).\end{equation}
From \eqref{eq:feast},
we deduce the bound \begin{equation}\label{eq:olivol}
\frac{1}{\log x} \sum_{n\leq x} \frac{\lambda(n) \lambda(n+1)}{n} = O\left(\frac{1}{\sqrt{\log \log x}}\right)\end{equation}
we stated at the beginning.

More generally, we get
\(\langle f,A f\rangle=O(\sqrt{\mathscr{L}})\) for any \(f\) with
\(|f|_\infty\leq 1\), or for that matter by any \(f\) with
\(|f|_4\leq e^{100 \mathscr{L}}\) and \(|f|_2\leq 1\). We obtain plenty
of consequences besides \eqref{eq:olivol}.

\subsection{Powers, eigenvalues and closed walks}

Now that we know what we want to prove, let us come up with a strategy.

There is a completely standard route towards bounds on eigenvalues of
operators such as \(A\) (or \(A|_{\mathscr{X}}\)), relying on the fact
that the trace is invariant under conjugation. Because of this
invariance, the trace of a power \(A^{2 k}\) is the same whether \(A\)
is written taking a full family of orthogonal eigenvectors as a basis,
or just taking the characteristic functions \(1_n\) as our basis.
Looking at matters the first way, we see that
\[\textrm{Tr} (A|_\mathscr{X})^{2 k} = \sum_{i=1}^N \lambda_i^{2 k},\]
where \(\lambda_1,\lambda_2,\dotsc,\lambda_N\) are the eigenvalues
corresponding to the basis made out of eigenvectors. Looking at matters
the second way, we see that
\(\textrm{Tr} (A|_{\mathscr{X}})^{2 k} = N_{2 k}\), where \(N_{2 k}\) is
the sum over all closed walks of length \(2 k\) of the products of the
weights of the edges in each walk:
\[N_{2 k} = \sum_{n\in \mathscr{X}} \sum_{\substack{p_1,\dotsc,p_{2 k}\in \mathbf{P}\\ \sigma_1,\dotsc,\sigma_{2 k}\in \{-1,1\}\\
\forall 1\leq i\leq 2 k:
n+\sigma_1 p_1 + \dotsc + \sigma_i p_i\in \mathscr{X} \\
\sigma_1 p_1 + \dotsc + \sigma_{2 k} p_{2 k} = 0}}
\prod_{i=1}^{2 k} \left(1_{p_i|n+\sigma_1 p_1+\dotsc+\sigma_{i-1} p_{i-1}} - \frac{1}{p_i}\right)\]
where we adopt the convention \(1_\text{true}=1\), \(1_\text{false}=0\).

Since all eigenvalues are real, it is clear that
\[\lambda_i^{2 k} \leq N_{2 k}\] for every eigenvalue \(\lambda_i\).
Often, and also now, that inequality is not enough in itself for a good
bound on \(\lambda_i\). What is then often done is to show that every
eigenvalue must have multiplicity \(\geq M\), where \(M\) is some large
quantity. Then it follows that, for every eigenvalue \(\gamma\),
\[M \gamma^{2 k} \leq N_{2 k},\] and so
\(|\gamma|\leq (N_{2 k}/M)^{1/2 k}\).

We do not quite have high multiplicity here (why would we?) but we have
something that is almost as good: if there is one large eigenvalue, then
there are many mutually orthogonal functions \(g_i\) of norm \(1\) with
\(\langle g_i, A g_i\rangle\) large. Then we can bound
\(\textrm{Tr} A^{2 k}\) from below, using these functions \(g_i\) (and
some arbitrary functions orthogonal to them) as our basis, and, since
\(\textrm{Tr} A^{2 k}\) also equals \(N_{2 k}\), we can hope to obtain a
contradiction with an upper bound on \(N_{2 k}\).

For simplicity, let us start by sketching a proof that, if
\(|\langle f, A f\rangle|\) is large (\(\geq \rho \mathscr{L}\), say)
for some \(f\) with \(|f|_\infty\leq 1\), then there are many orthogonal
functions \(g_i\) of norm \(1\) and \(\langle g_i, A g_i\rangle\) large
(with ``large'' meaning \(\geq \rho \mathscr{L}/2\), say). This weaker statement
suffices for our original goal, since we may set \(f\) equal to the
Liouville function \(\lambda\).

Let \(I_1,I_2,\dotsc \subset \mathbb{N}\) be disjoint intervals of
length \(\geq 10 H/\rho\) (say) covering \(\mathbb{N}\). Edges starting
at a vertex \(v\) in \(I_i\) end at another vertex in \(I_i\), unless
they are close to the edge. Hence,
\(\sum_i |\langle f|_{I_i},A \left(f|_{I_i}\right)\rangle|\) is not much
smaller than \(|\langle f,A f\rangle|\), and then it follows easily that
\(\langle f|_{I_i}, A \left(f|_{I_i}\right)\rangle/|f|_{I_i}|_2^2\) must
be large for many \(i\). Thus, setting \(g_i = f|_{I_i}/|f|_{I_i}|\) for
these \(i\), we obtain the desired statement.

To prove truly that \(A\) has no large eigenvalues, we should proceed as
we just did, but assuming only that \(|f|_2\leq 1\), not that
\(|f|_\infty \leq 1\). The basic idea is the same, except that (a)
pigeonholing is a little more delicate, (b) if \(f\) is almost entirely
concentrated in a small subset of \(\mathbf{N}\), then we can extract
only a few mutually orthogonal functions \(g_i\) from it. Recall that we
are anyhow restricting to a set \(\mathscr{X}\subset \mathbf{N}\). A
brief argument suffices to show that we can avoid the problem posed by
(b) simply by making \(\mathscr{X}\) a little smaller (essentially:
deleting the support of such \(g_i\), and then running through the
entire procedure again), while keeping its complement
\(\mathbf{N}\setminus \mathscr{X}\) very small.

In any event: we obtain that, if, for some \(X\subset \mathbf{N}\),
\(\textrm{Tr} (A|_X)^{2 k}\) is not too large (smaller than
\((\rho \mathscr{L}/2)^{2 k} N/H\) or so) then there is a subset
\(\mathscr{X}\subset X\) with \(X\setminus \mathscr{X}\) small such that
every eigenvalue of \(A|_\mathscr{X}\) is small
(\(\leq \rho \mathscr{L}\)). It thus remains to prove that
\(\textrm{Tr} (A|_X)^{2 k}\) is small for some \(X\subset \mathbf{N}\)
with small complement \(\mathbf{N}\setminus X\).

Recall that \(\textrm{Tr} (A|_X)^{2 k} = N_{2 k}\) (with \(N_{2 k}\)
defined as above, except with \(X\) instead of \(\mathscr{X}\)) and that
\(X\) should not include integers \(n\) with many more prime divisors in
\(\mathbf{P}\) than average. Our task is to bound \(N_{2 k}\).

\subsection{A brief look back}

We have come full circle, or rather we have arrived twice at the same
place. We started with a somewhat naïve approach that lead us to random
walks. Then we took a step back and analyzed the situation in a way that
turned out to be cleaner; for instance, the problem involving
\(\textrm{err}_2^{1/k}\) vanished. As it happens, that cleaner approach
took us to random walks again. Surely this is a good sign.

It is also encouraging to see signs that other people have thought in
the same direction. The paper by
\href{https://arxiv.org/abs/1509.01545}{Matomäki-Radziwiłł-Tao} on sign
patterns of \(\lambda\) and \(\mu\) is based on the examination of a
graph equivalent to \(\Gamma\); what they show is, in essence, that
\(\Gamma\) is almost everywhere locally connected. Being connected may
be a much weaker property than expansion, but it is a step in the same
direction. As for expansion itself,
\href{https://arxiv.org/abs/1509.05422}{Tao} (§4) comments that ``some
sort of expander graph property'' may hold for that graph (equivalent to
\(\Gamma\)) ``or {[}for{]} some closely related graph''. He goes on to
say:

\begin{quote}
Unfortunately we were unable to establish such an expansion property, as
the edges in the graph {[}\ldots{}{]} do not seem to be either random
enough or structured enough for standard methods of establishing
expansion to work."
\end{quote}

And so we will set about to establish expansion by our methods (standard
or not).

In any event, our initial discussion of random walks is still pertinent. Recall the plan with which we concluded, namely, to divide walks into three kinds: walks with few non-repeated primes, walks imposing many independent divisibility conditions, and rare walks. This plan will shape our approach to bounding $N_{2 k}$ in the next section.

\section{Main part of the proof: counting closed walks}

Let us recapitulate. Let
\(\mathbf{N} = \{n\in \mathbb{Z}: N < n\leq 2 N\}\). We have defined a
linear operator \(A\) on functions \(f:\mathbf{N}\to \mathbb{C}\) as the
difference of the adjacency operators of two graphs \(\Gamma\),
\(\Gamma'\): \[A = \textrm{Ad}_{\Gamma} - \textrm{Ad}_{\Gamma'}.\] We
would like to show that there is a subset \(X\subset \mathbf{N}\) with
small complement \(\mathbf{N}\setminus X\) such that, for some \(k\)
that is not too small, the trace \[\textrm{Tr} (A|_X)^{2 k}\] is
substantially smaller than \(\mathscr{L}^{2 k} N\). Indeed, we will
prove that there is a constant \(C\) such that
\[\textrm{Tr} (A|_X)^{2 k} \leq (C \mathscr{L})^{k} N,\] where
\(\mathscr{L} = \sum_{p\in \mathbf{P}} 1/p\).

Incidentally, when we
say ``\(k\) not too small'', we mean ``\(k\) is larger than \(\log H\)
or so''; we already saw that we stand to lose a factor of \(H^{1/k}\)
when going from (a) a trace bound as above to (b) a bound on
eigenvalues, which is our ultimate goal. If \(k\gg \log H\), then
\(H^{1/k}\) is just a constant.

For comparison: if, as will be the case, we define \(X\) so that every
\(n\in X\) has at most \(K \mathscr{L}\) prime factors, the trivial
bound is \[\textrm{Tr} (A|_X)^{2 k} \leq ((K+1) \mathscr{L})^{2 k} N.\]

We also saw that \(\textrm{Tr} (A|_X)^{2 k}\) can be expressed as a sum
over closed walks, i.e., walks that end where they start:
\[\textrm{Tr} (A|_X)^{2 k} = \sum_{n\in X} \sum_{\substack{p_1,\dotsc,p_{2 k}\in \mathbf{P}\\ \sigma_1,\dotsc,\sigma_{2 k}\in \{-1,1\}\\
\forall 1\leq i\leq 2 k:\;
n+\sigma_1 p_1 + \dotsc + \sigma_i p_i\in X\\
\sigma_1 p_1 + \dotsc + \sigma_{2 k} p_{2 k}=0}}
\prod_{i=1}^{2 k} \left(1_{p_i|n+\sigma_1 p_1+\dotsc+\sigma_{i-1} p_{i-1}} - \frac{1}{p_i}\right).\]
Here the double sum just goes over closed walks of length \(2 k\) in the
weighted graph \(\Gamma - \Gamma'\), which has \(X\) as its set of
vertices and an edge between any two vertices \(n,n'\) whose difference
\(n'-n\) is a prime \(p\) in our set of primes \(\mathbf{P}\); the
weight of the edge is then \(1-1/p\) if \(p|n\), and \(-1/p\) otherwise.
The contribution of a walk equals the product of the weights of its
edges.

\begin{center}
\begin{tikzpicture}[thick, scale=0.85, every node/.style={transform shape}]
\tikzstyle{membre}= [rectangle]
\tikzstyle{operation}=[->,>=latex]
\tikzstyle{etiquette}=[midway,fill=black!20]
\node[membre] (n) at (0,3) {$n$};
\node[membre] (n1) at (2.5,4) {$n_1=n+\sigma_1 p_1$};
\node[membre] (n2) at (6,4.5) {$n_2 = n+\sigma_1 p_1 + \sigma_2 p_2$};
\node[membre] (aro) at (9.5,4) {$\dots$};
\node[membre] (m) at (12,3) {$n_k = n+\sigma_1 p_1 + \dots + \sigma_k p_k$};
\node[membre] (nk1) at (8.5,1) {$n_{k+1} = n+\sigma_1 p_1 + \dots + \sigma_k p_k + \sigma_{k+1} p_{k+1}$};
\node[membre] (nk2m1) at (2.75,1) {$\dots$};
\draw[operation] (n) to (n1); 
\draw[operation] (n1) to  (n2);
\draw[operation] (n2) to (aro);
\draw[operation] (aro) to(m);
\draw[operation] (m) to (nk1);
\draw[operation] (nk1) to (nk2m1);
\draw[operation] (nk2m1) to (n);
\end{tikzpicture}
\end{center}
\subsection{Cancellation}

It might be nicer to work with an expression with yet simpler weights.
First, though, let us see what gains we can get from cancellation. Let
\(p_1,\dotsc,p_{2 k}\in \mathbf{P}\) and
\(\sigma_1,\dotsc,\sigma_{2 k}\in \{-1,1\}\) be given, and consider the
total contribution of the paths they describe as \(n\) varies in \(X\).
Say there is a \(p_i\) that appears only once, i.e., \(p_j\ne p_i\) for
all \(j\ne i\). The weight of the edge from
\(n_{i-1} = n + \sigma_1 p_1 + \dotsc + \sigma_{i-1} p_{i-1}\) to
\(n_i = n + \sigma_1 p_1 + \dotsc + \sigma_i p_i\) is \(1-1/p\) if
\(p|n_{i-1}\) and \(1/p\) otherwise. The weights of all the other edges
depend on the congruence classes \(n \;\textrm{mod}\; p_j\) for all
\(j\ne i\).

Suppose for a moment that \(X = \mathbf{N}\). Then, for \(\vec{p}\),
\(\vec{\sigma}\) fixed, and \(n\) in a given congruence class
\(n \;\textrm{mod}\; p_j\) for every \(j\ne i\) (that is, \(n\) in a
given congruence class \(a +P \mathbb{Z}\) for
\(P = \prod_{p\in \{p_1,.\dotsc,p_{i-1}.p_{i+1},\dotsc,p_{2 k}\}} p\),
by the Chinese remainder theorem), the probability that \(p_i\) divides
\(n_{i-1}\) is almost exactly \(1/p_i\): the number of \(n\) in
\(\mathbf{N}\) in our congruence class \(\textrm{mod}\; P\) is
\(N/P + O^*(1)\) (that is, no less than \(N/P-1\) and no more than
\(N/P+1\)), and, for such \(n\), again by the Chinese remainder theorem,
\(p|n_{i-1}\) if and only if \(n\) lies in a certain congruence class
modulo \(p_i\cdot P\); the number of \(n\) in \(\mathbf{N}\) in that
congruence class is \(N/(p_i P) + O^*(1)\).

Hence, among all \(n\) in \(\mathbf{N} \cap (a+ P \mathbb{Z})\), a
proportion almost exactly \(1/p\) have a weight \(1-1/p\) on the edge
from \(n_{i-1}\) to \(n_i\), and a proportion almost exactly \(1-1/p\)
have a weight \(-1/p\) there instead. Since all other weights are fixed,
we obtain practically total cancellation:
\[\frac{1}{p} \left(1 - \frac{1}{p}\right)
- \left(1 - \frac{1}{p}\right) \frac{1}{p}  = 0.\] In other words, the
contribution of paths where at least one \(p_i\) appears only once is
practically nil. Hence, we can assume that, in our paths, every \(p_i\)
appears at least twice among \(p_1,p_2,\dotsc,p_{2 k}\).

Of course we do not actually want to set \(X=\mathbf{N}\), and in fact
we cannot, as we have already seen. If \(X\) is well-distributed in
arithmetic progressions, then we should still get cancellation, but it
will not be total -- there will be an error term. Much of the pain here
comes from the fact that we have to exclude numbers with too many prime
factors (meaning: \(> K \mathscr{L}\) prime factors). Suppose for
simplicity that \(X\) is the set of all numbers in \(\mathbf{N}\) with
\(\leq K \mathscr{L}\). Recall that all vertices \(n\),
\(n_1 = n+\sigma_1 p_1\),
\(n_2 = n + \sigma_1 p_1 +\sigma_2 p_2 +\dotsc\) have to be in \(X\); in
particular, \(n_{i-1}\in X\). As a consequence, the likelihood that
\(p|n_{i-1}\) is slightly lower than \(1/p\): if \(n_{i-1} = p m\), then
\(m\) is constrained to have \(\leq K \mathscr{L}-1\) prime factors, and
it is slightly more difficult for \(m\) to satisfy that constraint than
it is for an \(n\in \mathbf{N}\) to have \(\leq K \mathscr{L}\) prime
factors. We do have cancellation, but it is not total, as it is for
\(X=\mathbf{N}\). The techniques involved in estimating how much
cancellation we do have are standard within analytic number theory.

Later, we will also exclude some other integers from \(X\), besides
those having \(> K \mathscr{L}\) prime factors. We will then need to
show that the effect on cancellation is minor. Doing so will require
some arguably new techniques; we will cross that bridge when we come to
it.

To cut a long story short, the effect of cancellation will be, not that
every \(p_i\) appears at least twice among \(p_1,p_2,\dotsc,p_{2 k}\),
but that the number of ``singletons'' (primes that appear only once) is
small. More precisely, a path with \(m\) singletons will have to pay a
penalty of a factor of \(\mathscr{L}^{-m/2}\).

\subsection{Shapes. Geometry of numbers and ranks}

Let us see what we have. Write \(\mathbf{k} = \{1,2,\dotsc,2k\}\). Let
\(\mathbf{l}\) range among all subsets of \(\mathbf{k}\) . Here
\(\mathbf{l}\) will be our set of ``lit'' indices, corresponding to the
set of indices \(i\) such that
\(p_i|n+\sigma_1 p_1 +\dotsc + \sigma_{i-1} p_{i-1}\) in the above.
Every ``unlit'' index \(i\) gives us a weight of \(1/p_i\). We define an
equivalence relation \(\sim\) on \(\mathbf{k}\) by letting \(i\sim j\)
if and only if \(p_i=p_j\). Given an equivalence class \([i]\), we
define \(p_{[i]}\) to equal \(p_i\) for any (and hence every)
\(i\in [i]\). If an equivalence class \([i]\) is not completely unlit
(that is, if \([i]\cap \mathbf{l}\ne \emptyset\)), then it gives us a
weight of \(1/p_{[i]}\) (coming from
\(p_i|n+\sigma_1 p_1 + \dotsc + \sigma_{i-1} p_{i-1}\) for some lit
index \(i\in [i]\)). It is also the case that, when two indices
\(i\sim j\) are both lit, they impose the condition
\[p_{[i]}|\sigma_{i+1} p_{i+1} + \dotsc + \sigma_j p_j,\] coming from
\(p_i|n + \sigma_1 p_1 + \dotsc + \sigma_i p_i\) and
\(p_i=p_j|n+\sigma_1 p_1 + \dotsc + \sigma_j p_j\). Let us write
\(\beta_i\) as shorthand for \(\sigma_1 p_1 + \dotsc + \sigma_i p_i\);
then our condition becomes \[p_{[i]}|\beta_j-\beta_i.\]

Given a walk \(n, n+\sigma_1 p_1, n+\sigma_1 p_1 +\sigma_2 p_2,\dotsc\),
we define its \emph{shape} to be \((\sim, \vec{\sigma})\), where
\(\sim\) is the equivalence relation it induces (as above). In fact, let
us start with shapes, meaning pairs \((\sim, \vec{\sigma})\), where
\(\sim\) is an equivalence class on \(\{1,2,\dotsc,k\}\) and
\(\vec{\sigma}\in \{-1,1\}^{2 k}\). For any given shape, we will bound
the contribution of all walks of that shape. There will be some shapes
for which we will not be successful; we will later treat walks of those
shapes, and show that their contribution is small in some other way.

To rephrase what we said just before: given
\(\mathbf{l}\subset \mathbf{k}\), the contribution of a shape
\((\sim, \vec{\sigma})\) will be at most
\begin{equation}\label{eq:littlestar}\mathscr{L}^{-\frac{|\mathcal{S}(\sim)|}{2}}
\sum_{\substack{\{p_{[i]}\}_{[i]\in \Pi}, p_{[i]}\in \mathbf{P}\\i_1\sim i_2 \wedge (i_1,i_2\in \mathbf{l})\Rightarrow p_{i_1}|\beta_{i_2}-\beta_{i_1}}}
\prod_{i\not\in \mathbf{l}} \frac{1}{p_{[i]}}
\prod_{\substack{[i]\in \Pi\\ [i]\not\subset \mathbf{k}\setminus \mathbf{l}}} \frac{1}{p_{[i]}}
,\end{equation}
where \(\Pi\) is the set of
equivalence classes of \(\sim\) and \(\mathcal{S}(\sim)\) is the set of
singletons of \(\sim\), where a ``singleton'' is an equivalence class
with exactly one element. We write \(|S|\) for the number of elements of
a set \(S\).

What we have to do then is, in essence, bound the number of solutions
\((p_{[i]})_{[i]\in \Pi}\) to a system of divisibility conditions
\begin{equation}\label{eq:ddagger}
p_{[i]}|\sigma_{i+1} p_{[i+1]} + \dotsc + \sigma_j p_{[j]}.
\end{equation}
It would be convenient if the divisors \(p_{[i]}\) were all distinct
from the primes in the sums being divided. Then we could apply directly
the following Lemma, which is really grade-school-level geometry of
numbers.

\begin{lemma}
Let \(\mathbf{M}=(b_{i,j})_{1\leq i,j\leq m}\) be a
non-singular \(m\)-by-\(m\) matrix with integer entries. Assume
\(|b_{i,j}|\leq C\) for all \(1\leq i,j\leq m\). Let
\(\vec{c}\in \mathbb{Z}^m\), and let \(d_1,\dotsc,d_m\geq D\), where
\(D\geq 1\). Let \(N_1,\dotsc,N_m\) be real numbers \(\geq D\). Then the
number of solutions \(\vec{n}\in \mathbb{Z}^m\) to
\[d_i|(\mathbf{M} \vec{n} + \vec{c})_i\;\;\;\;\;\forall 1\leq i\leq m\]
with \(N_i\leq n_i\leq 2 N_i\) is at most
\[\left(\frac{2 C m}{D}\right)^m \prod_{i=1}^m N_i.\]
\end{lemma}

The trivial bound is clearly \(\prod_{i=1}^m (N_i+1)\).

\begin{proof}
Divide the box \(\prod_{i=1}^m [N_i,2 N_i]\) into
\(\leq \prod_{i=1}^m \left(\frac{N_i}{D}+1\right)\leq \left(\frac{2}{D}\right)^m \prod_{i=1}^m N_i\)
\(m\)-dimensional boxes of side \(D\). The image of such a box under the
map \(\vec{n}\mapsto \mathbf{M} \vec{n} + \vec{c}\) is contained in a
box whose edges are open or half-open intervals of length \(C m D\).
Since \(d_i\geq D\), that box contains at most \((C m)^m\) solutions
\(\vec{m}\) to the equations \(d_i|m_i\).
\end{proof}

Of course we \emph{can} make our set of divisors and our set of
variables disjoint: we can choose to color some equivalence classes
\([i]\) blue and some other equivalence classes \([j]\) red, and
consider only those divisibility relations \eqref{eq:ddagger} in which
\([i]\) is colored blue. We fix \(p_{[i]}\) for \([i]\) blue, and in
fact for all non-red \([i]\), and treat \(p_{[j]}\) with \([j]\) as our
variables. We can then use the Lemma above to bound the number of values
of \((p_{[j]})_{[j]\; \text{red}}\) that satisfy our divisibility
relations.

To be precise: let \(x_{[j]}\) be a formal variable for each red
equivalence class \([j]\). Define
\begin{equation}\label{eq:cross}
v(i) = \sum_{j < i:\; [j]\;\text{red}} \sigma_j x_{[j]}.\end{equation}
Let \(r\) be the dimension of the space spanned by the differences
\(v(i_2)-v(i_1)\) with \([i_1]=[i_2]\) blue. Then we can select \(r\)
divisibility relations of the form \eqref{eq:ddagger} and \(r\) red
equivalence classes such that the matrix consisting of a row
\((\sum_{j\in [j]} \sigma_j)_{[j]\;\text{red}}\) for each equivalence
relation is non-singular. (We are just saying that a matrix of rank
\(r\) has a non-singular \(r\)-by-\(r\) submatrix.) We can then apply
our lemma.

After some book-keeping, we obtain a bound on our sum from \eqref{eq:littlestar}, namely,
\begin{equation}\label{eq:circ}\sum_{\substack{\{p_{[i]}\}_{[i]\in \Pi},\; p_{[i]}\in \mathbf{P}\\i_1\sim i_2 \wedge (i_1,i_2\in \mathbf{l})\Rightarrow p_{i_1}|\beta_{i_2}-\beta_{i_1}}}
\prod_{i\not\in \mathbf{l}} \frac{1}{p_{[i]}}
\prod_{\substack{[i]\in \Pi\\ [i]\not\subset \mathbf{k}\setminus \mathbf{l}}} \frac{1}{p_{[i]}} \leq \frac{1}{H_0^r}
\left(\frac{4 k r \log H}{\mathscr{L} \log 2}\right)^r \mathscr{L}^{|\Pi|}.\end{equation}
Here the important factor is \(1/H_0^r\). We see that we ``win'' if
\(r\) is at least somewhat large. The question is then how to choose
which equivalence classes to color red or blue so as to make the rank
\(r\) large.

\subsection{Ranks and a new graph. Sets with large boundary}

To address this question, let us define a new graph. First, though, let
us define the \emph{reduction} of a shape \((\sim, \sigma)\). A shape
clearly induces a word
\[w = x_{[1]}^{\sigma_1} x_{[2]}^{\sigma_2} \dotsc x_{[2 k]}^{\sigma_{2 k}}.\]
This word can be reduced (if it isn't already), and the resulting
reduced word induces a ``reduced shape'' \((\sim',\sigma')\). If all
representatives of an equivalence class of \((\sim,\sigma)\) disappear
during the reduction, we color that class yellow. It is the non-yellow
classes that we will color red or blue.

We define a graph \(\mathscr{G}_{(\sim,\sigma)}\) to be an undirected
graph having the non-yellow equivalence classes as its vertices, and an
edge between two vertices \(v_1\), \(v_2\) if there are \(i_1\in v_1\),
\(i_2\in v_2\) such that every equivalence class containing at least one
index \(j\in \{i_1+1,i_1+2,\dotsc,i_2-1\}\) is yellow.

(We define matters in this way, rather than simply reduce the word and
join two vertices \(v_1\), \(v_2\) if there are \(i_1\in v_1\),
\(i_2\in v_2\) such that \(i_2 = i_1+1\), because reducing the word
could create more singletons. At any rate, the idea is that, if there
are only yellow indices between \(i_1\) and \(i_2\), then
\(v(i_1) = v(i_2)\), where \(v(i)\) is defined as in \((+)\). But we are
getting ahead of ourselves.)

Let us see two examples. Let \(k=3\), and let \(\sim\) have equivalence
classes \[\{\{1,4\},\{2,5\},\{3\},\{6\}\},\] with
\(\sigma\in \{-1,1\}^{2 k}\) arbitrary. Then the graph
\(\mathscr{G}_{(\sim,\sigma)}\) is

\begin{center}
\begin{tikzpicture}[thick,scale=0.75]
\tikzstyle{every node}=[circle, draw, fill=white!50,inner sep=0pt, minimum width=4pt]
 \draw {
 (0:0) node {$\{1,4\}$} -- (0:6) node {$\{2,5\}$}
 (0:6) -- (-40:4) node {$\{3\}$}
 (0:0) -- (-40:4)
 (0:6) -- (-10:10) node {$\{6\}$}
 };
\end{tikzpicture}
\end{center}

As for our second example, let \(k = 5\),
\(\vec{\sigma}=(1,-1,1,-1,1,-1,1,1,1,-1)\), and let \(\sim\) have
equivalence classes \(\{\{1,8\},\{2,9\},\{3\},\{4,7\},\{5,6,10\}\}\).
Then the induced word is
\[w = x_{[1]} x_{[2]}^{-1} x_{[3]} x_{[4]}^{-1} x_{[5]} x_{[5]}^{-1} x_{[4]} x_{[1]} x_{[2]} x_{[5]}^{-1},\]
which has reduction
\[w' = x_{[1]} x_{[2]}^{-1} x_{[3]} x_{[1]} x_{[2]} x_{[5]}^{-1}.\]
Hence, the equivalence class \(\{4,7\}\) is colored yellow, and the
graph \(\mathscr{G}_{(\sim,\vec{\sigma})}\) is

\begin{center}
\begin{tikzpicture}[thick,scale=0.75]
\tikzstyle{every node}=[circle, draw, fill=white!50, inner sep=0pt, minimum width=4pt]
 \draw {
 (0:0) node {$\{1,8\}$} -- (0:6) node {$\{2,9\}$}
 (0:6) -- (-40:4) node {$\{3\}$}
 (0:0) -- (-40:4)
 (0:6) -- (-10:10) node {$\{5,6,10\}$}
 (0:0) -- (-10:10)
 (-40:4) -- (-10:10)
 };
\end{tikzpicture}
\end{center}
It is clear from the definition that
\(\mathscr{G}_{(\sim,\vec{\sigma})}\) is always connected: if
\(i_1<i_2<i_3<\dotsc\) are the indices in non-yellow equivalence
classes, then there is an edge from \([i_1]\) to \([i_2]\), an edge from
\([i_2]\) to \([i_3]\), etc.

Given a subset \(V'\) of the set of vertices \(V\) of a graph
\(\mathscr{G}\), write \(\mathscr{G}|_{V'}\) for the restriction of
\(\mathscr{G}\) to \(V'\), i.e., the subgraph of \(\mathscr{G}\) having
\(V'\) as its set of vertices and set of all edges in \(\mathscr{G}\)
between elements of \(V'\) as its set of edges. What happens if we
choose our coloring so that the restriction
\(\mathscr{G}|_{\textbf{blue}}\) to the set of blue vertices (named
\(\textbf{blue}\)) is connected?

\begin{lemma}
Let \((\sim,\sigma)\) be a shape, and let
\(\mathscr{G}=\mathscr{G}_{(\sim,\vec{\sigma})}\) and \(v(i)\) be as
above. Color some non-yellow vertices red and some other non-yellow
vertices blue, in such a way that, for \(\textbf{blue}\) the set of blue
vertices, the restriction \(\mathscr{G}|_\textbf{blue}\) is connected.

Then the space \(V\) spanned by the vectors
\[v(i_2)-v(i_1)\;\;\;\;\;\;\;\;\text{with}\;\;\; [i_1]=[i_2]\; \textbf{blue}\]
equals the space \(W\) spanned by all vectors
\[v(i_2)-v(i_1)\;\;\;\;\;\;\;\;\text{with}\;\;\; [i_1], [i_2]\; \text{both \bf{blue}}.\]
\end{lemma}

The proof is an exercise, and its idea may be best made clear by an
example.
\begin{proof}[Sketch of proof (or rather, a worked example)]
Say we have three blue equivalence classes, corresponding to
letters \(x\), \(y\), \(z\) in the induced word, and let them be
disposed as follows:
\[\underbrace{\;\;\;\;\;\;\;\;\;\;\;\;\;\;} \textcolor{blue}{x} \underbrace{\;\;\;\;\;\;\;\;\;\;\;\;\;\;\;\;\;\;} \textcolor{blue}{z} \underbrace{} \textcolor{blue}{y x} \underbrace{\;\;\;\;\;\;\;\;\;} \textcolor{blue}{z y} \underbrace{}\]
Call the indices of the six letters we have written
\(i_1,i_2,\dotsc,i_6\). Then the space \(V\) in the Lemma is the space
spanned by \[v_{i_4}-v_{i_1},\; v_{i_6}-v_{i_3},\; v_{i_5}-v_{i_2},\]
where the space \(W\) in the Lemma is the space spanned by all vectors
\[v_{i_r} - v_{i_{s}},\;\;\;\;\;\;\;\;\;\; 1\leq r,s\leq 6.\] It is
clear that \(V\subset W\), but why is \(W\subset V\)? Why, say, is
\(v_{i_2}-v_{i_1}\) in \(V\)? Well, let us follow a path in
\(\mathscr{G}|_{\textbf{blue}}\) going from \(x\) (the first blue
letter, i.e., the letter at position \(i_1\)) to \(z\) (the second blue
letter, i.e., the letter at position \(i_2\)): there is an edge from
(the equivalence class labeled) \(x\) to (the equivalence class
label-led) \(y\), and an edge from \(y\) to \(z\). So:

\begin{longtable}[]{@{}ll@{}}
\toprule
\endhead
\begin{minipage}[t]{0.47\columnwidth}\raggedright
\(v_{i_4}-v_{i_1}\) is in \(V\)\strut
\end{minipage} & \begin{minipage}[t]{0.47\columnwidth}\raggedright
because \(i_1\) and \(i_4\) are both in the equivalence class
\(x\)\strut
\end{minipage}\tabularnewline
\begin{minipage}[t]{0.47\columnwidth}\raggedright
\(v_{i_4}\) equals \(v_{i_3}\)\strut
\end{minipage} & \begin{minipage}[t]{0.47\columnwidth}\raggedright
because \(i_4\) and \(i_3\) are adjacent (meaning there cannot be red
indices between them; note that there is an edge from \(x\) to \(y\)
precisely because \(i_4\) and \(i_3\) are adjacent)\strut
\end{minipage}\tabularnewline
\begin{minipage}[t]{0.47\columnwidth}\raggedright
\(v_{i_6}-v_{i_3}\) is in \(V\)\strut
\end{minipage} & \begin{minipage}[t]{0.47\columnwidth}\raggedright
because \(i_6\) and \(i_3\) are both in \(y\)\strut
\end{minipage}\tabularnewline
\begin{minipage}[t]{0.47\columnwidth}\raggedright
\(v_{i_5}\) equals \(v_{i_6}\)\strut
\end{minipage} & \begin{minipage}[t]{0.47\columnwidth}\raggedright
because \(i_6\) and \(i_5\) are adjacent, as is again reflected in the
fact that there is an edge from \(y\) to \(z\),\strut
\end{minipage}\tabularnewline
\begin{minipage}[t]{0.47\columnwidth}\raggedright
\(v_{i_2}-v_{i_6}\) is in \(V\)\strut
\end{minipage} & \begin{minipage}[t]{0.47\columnwidth}\raggedright
because \(i_2\) and \(i_6\) are both in \(z\).\strut
\end{minipage}\tabularnewline
\bottomrule
\end{longtable}

Hence, \(v_{i_2}-v_{i_1}\in V\), as we have shown by following a path in
\(\mathscr{G}|_{\textbf{blue}}\) from \(x\) to \(z\). The same argument
works in general for any two indices in blue equivalence classes.
\end{proof}

Now we have to bound the rank of \(W\) from below. We first reduce our
word; yellow letters disappear. The most optimistic expectation would be
that the rank of \(W\) equal the number of gaps between blue ``chunks''
indicated by braces in our example from before:
\[\underbrace{\;\;\;\;\;\;\;\;\;\;\;\;\;\;} \textcolor{blue}{x} \underbrace{\;\;\;\;\;\;\;\;\;\;\;\;\;\;\;\;\;\;} \textcolor{blue}{z} \underbrace{} \textcolor{blue}{y x} \underbrace{\;\;\;\;\;\;\;\;\;} \textcolor{blue}{z y} \underbrace{}\]
(Chunks may have merged during reduction.) The number of gaps here is
\(4\), considered cyclically (so that the first and last gap become
one). The gaps correspond to \(v(i_2)-v(i_1)\), \(v(i_3)-v(i_2)\),
\(v(i_5)-v(i_4)\) and (lastly, or firstly) \(v(i_1)-v(i_6)\).

Imagine for a moment that each red letter appeared only once. (From now
on, all letters that are neither blue nor yellow will be colored red. In
our reduced word, all letters that are not blue are red.) Then the
optimistic expectation would hold: each of
\(v(i_2)-v(i_1), v(i_5)-v(i_4),\dotsc,v(i_1)-v(i_6)\) would be a
non-trivial formal linear combination of a non-zero number of symbols
\(x_{[j]}\), each appearing only once altogether, and so those
combinations must all be linearly independent.

Of course, we cannot ensure that each red letter will appear only once,
and in fact we are usually treating cases where most of them appear at
least twice (i.e., singletons are rare). Let us see what we can do with
a weaker assumption. What if we assume that each red letter appears at
most \(\kappa\) times?

Let us see an easy linear-algebra lemma.
\begin{lemma}
Let \(A\) be a matrix with \(n\)
rows, satisfying:

\begin{itemize}
\tightlist
\item
  every row has at least one non-zero entry,
\item
  no column has more than \(\kappa\) non-zero entries.
\end{itemize}

Then the rank of \(A\) is \(\geq n/\kappa\).
\end{lemma}
\begin{proof}
 We will construct a finite list \(S\) of columns, starting
with the empty list. At each step, if there is a row \(i\) such that the
\(i\)th entry of every column in \(S\) is \(0\), include at the end of
\(S\) a column whose \(i\)th entry is non-zero. Stop if there is no such
row.

When we stop, we must have \(\kappa\cdot |S|\geq n\), as otherwise there
would still be a row in which no element of \(S\) would have a non-zero
entry. Since, for each column in \(S\), there is a row in which that
column has a non-zero entry and no previous column in \(S\) does, we see
that the columns in \(S\) are linearly independent. Hence
\(\textrm{rank}(A)\geq |S|\geq n/\kappa\).
\end{proof}

Now we see what to do: let \(A\) be a matrix with columns corresponding
to red equivalence classes, and rows corresponding to gaps between blue
chunks, with the entry \(a_{i [j]}\) being the number of times the red
letter corresponding to column \([j]\) appears in the gap corresponding
to row \(i\) (counting appearances as \(x_{[j]}^{-1}\) as negative
appearances). Each column has no more than \(\kappa\) non-zero entries
because, by assumption, each red equivalence class contains at most
\(\kappa\) elements. We still need to show that no or few rows are full
of zeros.

A row is full of zeros iff every letter \(x\) in the corresponding gap
appears an equal number of times as \(x\) and as \(x^{-1}\) (i.e., every
equivalence class has as many representatives \(i\) with \(\sigma_i=1\)
as with \(\sigma_i=-1\) within the gap). Let us call such gaps
\emph{invalid}.

There is a condition that limits how many such gaps there can be while
at the same time ensuring that an equivalence class contains at most
\(\kappa\) elements (or not quite, but something that is as good). Let
\((\sim',\sigma')\) (of length \(2 k'\)) be the reduction of the shape
\((\sim,\sigma)\). Let us say that we see a \emph{revenant} when there
are indices \(i\), \(i'\) such that (a) \(i\sim i'\), and (b) there is a
\(j\not\sim i\) with \(i<j<i'\). (In other words, \(x_{[i]}\) has come
back after going away.) We say that there are \(\kappa\) \emph{disjoint
revenants} if there are
\[1\leq i_1<\jmath_1<i_1'\leq i_2<\jmath_2<i_2'\leq \dotsc \leq
 i_\kappa < \jmath_\kappa < i_\kappa'\leq 2 k'\] with \(i_j\sim' i_j'\)
and \(i_j\not\sim' \jmath_j\) for \(1\leq j\leq 2 k\). Thus, for
example, in \[xxz\dotsc x^{-1} y w \dotsc y v \dotsc y,\] we see three
disjoint revenants (with \(x\) at positions \(i_1\) and \(i_1'\), and
\(y\) at positions \(i_2\), \(i_2' = i_3\) and \(i_3'\)).

Let us impose the condition that there cannot be more than \(\kappa\)
disjoint revenants in our walk. (We will be able to assume this
condition by rigging the definition of \(X\) later.) Then it follows
immediately that the appearances of a letter form at most \(\kappa\)
contiguous blocks in the reduced word. Hence, a red letter cannot appear
in more than \(\kappa\) gaps. We also see that there cannot be more than
\(\kappa\) invalid gaps: a gap is a non-empty reduced subword, and, if a
letter \(x\) appears in a reduced, non-trivial word as many times as
\(x\) and as \(x^{-1}\), either the pattern \(x \dotsc x^{-1}\) or the
pattern \(x^{-1} \dotsc x\) appears in the word, with ``\(\dotsc\)''
standing for a non-empty subword consisting of letters that are not
\(x\). Thus, \(> \kappa\) invalid gaps would give us \(>\kappa\)
revenants, all disjoint.

It then follows, by the easy linear-algebra lemma above, that
\[\textrm{dim}(W)\geq \frac{s-\kappa}{\kappa} = \frac{s}{\kappa}-1,\]
where \(s\) is the number of gaps.

The question is then: how do you choose which letters to color blue and
which to color red so that the number \(s\) of gaps is large?

\subsection{Spanning trees and boundaries}

Let us first assume that there are no yellow letters in the non-reduced
word, as that is a somewhat simpler case. Then the number of gaps equals
the number of red letters \(x_i\) such that \(x_{i-1}\) is blue (or
\(x_{2 k}\) is blue, if \(i=1\)). That number is bounded from below by
\[\frac{1}{2} |\partial \textbf{blue}|,\] where
\(\partial \textbf{blue}\) is the set of all red equivalence classes
\([i]\) such that there is a blue equivalence class \([j]\) connected to
\([i]\) by an edge in \(\mathscr{G} = \mathscr{G}_{\sim,\sigma}\)
(meaning that \([j]\) contains an index \(j\) and \([i]\) contains an
index \(i\) such that \(i\) and \(j\) are separated only by yellow
letters; since there are no yellow letters, that means that \(i=j+1\) or
\(i=j-1\) (or one of \(i\), \(j\) is \(1\) and the other one is
\(2 k\))).

The question is then how to choose the set \(\textbf{blue}\) of
equivalence classes to be colored blue in such a way that
\(\partial \textbf{blue}\) is large. Here \(\textbf{blue}\) can be any
set of vertices such that \(\mathscr{G}|_\textbf{blue}\) is connected.
So, in general: given a connected undirected graph \(\mathscr{G}\), how
do we choose a set \(\textbf{blue}\) of vertices so that
\(\mathscr{G}|_\textbf{blue}\) is connected and
\(\partial \textbf{blue}\) is large?

A \emph{spanning tree} of a graph \(\mathscr{G}=(V,E)\) is a subgraph
\((V,E')\) (where \(E'\subset E\)) that is a tree (i.e., has no cycles)
and has the same set of vertices \(V\) as \(\mathscr{G}\). Given a
spanning tree of \(\mathscr{G}\), we can define \(\textbf{blue}\) to be
the set of internal nodes of \(\mathscr{G}\), that is, the set of
vertices that are not leaves. Then \(\textbf{blue}\) is connected, and
\(\partial \textbf{blue}\) equals the set of leaves. The question is
then: is there a spanning tree of \(\mathscr{G}\) with many leaves?

Here there is a result from graph theory that we can just buy off the
shelf.

\begin{prop}[Kleiman-West, 1991; see also Storer, 1981,
Payan-Tchuente-Xuong, 1984, and Griggs-Kleitman-Shastri, 1989] Let
\(\mathscr{G}\) be a connected graph with \(n\) vertices, all of degree
\(\geq 3\). Then \(\mathscr{G}\) has a spanning tree with \(\geq n/4+2\)
leaves.
\end{prop}

Using this Proposition, we prove:

\begin{cor}
Let \(\mathscr{G}\) be a connected graph such that
\(\geq n\) of its vertices have degree \(\geq 3\). Then \(\mathscr{G}\)
has a spanning tree with \(\geq n/4+2\) leaves.
\end{cor}

We omit the proof of the corollary, as it consists just of less than a
page of casework and standard tricks. Alternatively, we can prove it
from scratch in about a page by modifying Kleiman and West's proof.

(It is clear that some condition on the degrees, as here, is necessary;
a spanning tree of a cyclic graph (every one of whose vertices has
degree \(2\)) has no leaves.)

Before we go on to see what do we do with shapes \((\sim,\vec{\sigma})\)
such that \(\mathscr{G}_{(\sim,\vec{\sigma})}\) does \emph{not} have
many vertices with degree \(\geq 3\), let us remove the assumption that
there are no yellow letters.

So, let us go back to counting gaps between blue chunks. For any two
distinct non-yellow equivalence classes \([i]\), \([j]\), let us draw an
arrow from \([i]\) to \([j]\) if there are representatives \(i\in [i]\),
\(j\in [j]\) that survive in the reduced word, and such as that all
letters between \(i\) and \(j\) disappear during reduction. (If \(j<i\),
then ``between'' is to be understood cyclically, i.e., the letters
between \(i\) and \(j\) are those coming after \(i\) or before \(j\).)
We draw each arrow only once, that is, we do not draw multiple arrows.

For instance, in our example
\(w = x_{[1]} x_{[2]}^{-1} x_{[3]} x_{[4]}^{-1} x_{[5]} x_{[5]}^{-1} x_{[4]} x_{[1]} x_{[2]} x_{[5]}^{-1}\)
from before,

\begin{center}
\begin{tikzpicture}[thick]
\tikzstyle{vertex}=[circle, draw, fill=white!50, inner sep=0pt, minimum width=4pt]
\node[vertex] (1) at (0:0) {$\{1,8\}$};
\node[vertex] (2) at (0:6) {$\{2,9\}$};
\node[vertex] (3) at (-40:4) {$\{3\}$};
\node[vertex] (5) at (-10:10) {$\{5,6,10\}$};
 \draw {
 (1) -- (2)
 (2) -- (3)
 (1) -- (3)
 (2) -- (5)
 (1) -- (5)
 (3) -- (5)
};
\path (1) edge [->, >=latex, bend left=15, color=olive] (2);
\path (2) edge [->, >=latex, bend right=15, color=olive] (3);
\path (3) edge [->, >=latex, bend left=20, color=olive] (1);
\path (2) edge [->, >=latex, bend left=15, color=olive] (5);
\path (5) edge [->, >=latex, bend left=10, color=olive] (1);
\end{tikzpicture}
\end{center}

It is obvious that every vertex has an in-degree of at least \(1\).

For \(S\) a set of vertices, define the \emph{out-boundary}
\(\vec{\partial} S\) to be the set of all vertices \(v\) not in \(S\)
such that there is an arrow going from some element of \(S\) to \(v\).
Then, whether or not there are yellow letters, the number of red gaps
\(\underbrace{}\) in the reduced word is at least
\(|\vec{\partial} \textbf{blue}|\).

\begin{lemma} Let \(G\) be a directed graph such that every vertex has
positive in-degree. Let \(S\) be a subset of the set vertices of \(G\).
Then there is a subset \(S'\subset S\) with \(|S'|\geq |S|/3\) such
that, for every \(v\in S'\), there is an arrow from some vertex not in
\(S'\) to \(v\).
\end{lemma}
\begin{proof}
 The first step is to remove arrows until the in-degree of
every vertex is exactly 1. Then \(G\) is a union of disjoint cycles. If
all vertices in a cycle are contained in \(S\), we number its vertices
in order, starting at an arbitrary vertex, and include in \(S'\) the
second, fourth, etc. elements. If no vertices in a cycle are in \(S\),
we ignore that cycle. If some but not all vertices in a cycle are in
\(S\), the vertices that are in \(S\) fall into disjoint subsets of the
form \(\{v_1,\dotsc v_r\}\), where there is an arrow from some \(v\) not
in \(S\) to \(v_1\), and an arrow from \(v_i\) to \(v_{i+1}\) for
\(1\leq i\leq r-1\); then we include \(v_1,v_3,\dotsc\) in \(S'\).
\end{proof}

We let \(S\) be the set of leaves of our spanning tree, and define
\(\textbf{red}\) to be the set \(S'\) given by the Lemma;
\(\textbf{blue}\) is the set of all other non-yellow equivalence
classes. Then the number of gaps is \(\geq (n/4+2)/3\), where \(n\) is
the number of vertices of degree \(\geq 3\) in
\(\mathscr{G}_{(\sim,\vec{\sigma})}\). Hence, by our work up to now,
\[\textrm{dim}(W) \geq \frac{1}{\kappa} \frac{\frac{n}{4}+2}{3} -1 \geq \frac{n}{12 \kappa} - 1,\]
and so, if \(n\) is even modestly large, we win by a large margin: we
obtain a factor nearly as small as \(1/H_0^{\frac{n}{12 \kappa}}\) in
\eqref{eq:circ}.

\emph{Note.} Had we been a little more careful, we would have obtained a
bound of \(\dim(W)\geq \frac{n}{50 \log \kappa} - \frac{\kappa}{2}\) or
so. This improvement -- which involves drawing, and considering,
multiple arrows -- would affect mainly the allowable range of \(H_0\) in
the end. We will remind ourselves of the matter later.

\subsection{Shapes with low freedom. Writer-reader arguments.}

The question now is what to do with walks of shapes
\((\sim,\vec{\sigma})\) for which \(\mathscr{G}_{(\sim,\vec{\sigma})}\)
does not have many vertices of degree \(\geq 3\).

Let us first give an argument that is sufficient when the word given by
our walk is already reduced; we will later supplement it with an
additional argument that takes care of the reduction. Let
\(\mathbf{n}\subset \{1,2,\dotsc,2 k\}\) be the set of indices that
survive the reduction. It is enough to define an equivalence relation
\(\sim\) on \(\mathbf{n}\) to define the graph
\(\mathscr{G}_\sim = \mathscr{G}_{\sim,\sigma}\) we have been
considering. (We do not need to specify \(\vec{\sigma}\), as its only
role was to help determine which letters are yellow.) Assume that
\(\mathscr{G}_\sim\) has \(\leq \nu\) vertices of degree \(\geq 3\). Let
\(\kappa\) be, as usual, an upper bound on the number of disjoint
revenants; in particular, for any equivalence class \([i]\), there are
at most \(\kappa\) elements \(i\in [i]\) such that the following element
of \(\mathbf{n}\) is not in \([i]\). We claim that the number of
equivalence classes \(\sim\) on \(\mathbf{n}\) satisfying these two
constraints (given by \(\nu\) and \(\kappa\)) is
\[\leq 5^{|\mathbf{n}|} (2 k)^{(\kappa-1) \nu + 2}.\]

We will prove this bound by showing that we can determine an equivalence
class of this kind by describing it by a string \(\vec{s}\) on \(5\)
letters with indices in \(\mathbf{n}\), together with some additional
information at each of at most \((\kappa-1) \nu+2\) indices. The idea is
that, if an index lies in an equivalence class that is a vertex of
degree \(1\) or \(2\) in \(\mathscr{G}_\sim\), then there are very few
possibilities for the equivalence classes on which the index just
thereafter may lie, namely, \(1\) or \(2\) possibilities.

We let the index \(i\) go through \(\mathbf{n}\) from left to right. If
\([i]\) is in an equivalence class we have not seen before, we let
\(s_i = *\). Assume otherwise. Let \(i_-\) be the element of
\(\mathbf{n}\) immediately preceding \(i\). If \([i]=[i_-]\), let
\(s_i=0\). If \([i_-]\) is a vertex of degree \(\leq 2\) and \(i\) is in
an equivalence class that we have already seen next to \([i_-]\) (that
is, just before or just after \([i_-]\) in \(\mathbf{n}\)), then we let
\(s_i=1\) or \(s_i=2\) depending on which one of those \(\leq 2\)
equivalence classes we mean (the first one or the second one to appear).
In all remaining cases, we let \(s_i = \cdot\), and specify our
equivalence class explicitly, by giving an index \(j<i\) in the same
equivalence class.

Let us give an example. Let \(k= 8\),
\(\mathbf{n} = \{1,2,\dotsc,2 k\}\). Let our equivalence classes be
\[\{1,7,15\},\{2,16\},\{3,4,5,11\},\{6,10,12,14\},\{8\},\{9,13\}.\] Then
\(s_1=s_2=s_3=s_6=s_8=s_9=*\) and \(s_4=s_5=0\). The vertices of degree
\(3\) are \([1]\) and \([6]\); all other vertices are of degree \(2\).
Hence, \(s_{16}=\cdot\) (since \(16\) follows \(15\), which is in
\([1]\)) and \(s_7=s_{11}=s_{13}=s_{15}=\cdot\) (since these indices
follow \(6, 10, 12, 14\), which are in \([6]\)). Since
\(3\sim 4\sim 5,\) we let \(s_4 = s_5 = 0\). It remains to consider
\(i=10,12,14\). In the case \(i=10\), we see that \([9]\) has degree
\(2\), but, when we come to \(10\), we realize that no element of
\([10]\) has been seen next to an element of \([9]\) before: \(8\) is
next to \(9\), but \(8\notin [10]\). Hence, we let \(s_{10}=\cdot\). In
the case \(i=12\), we see that \([12]\) has been seen next to \([11]\)
before: \(5\in [11]\) and \(6\in [12]\). Since \([12]\) was the second
equivalence class other than \([11]\) to appear next to \([11]\) (the
first one was \([2]\): \(2\in [2]\), \(3\in [11]\)), we write
\(s_{12}=2\). The situation for \(i=14\) is analogous, in that \([14]\)
appeared next to \([13]\) before: \(9\in [13]\), \(10\in [14]\), and so,
since \(8\notin [13],[14]\), \(s_{14}=2\).

In summary,
\[\vec{s} = \text{***00*$\cdot$**$\cdot\cdot$2$\cdot$2$\cdot\cdot$},\]
and, in addition to writing \(\vec{s}\), we specify the equivalence
classes of the indices \(i\) with \(s_i=.\) explicitly (\([1]\) for
\(i=7,15\), \([6]\) for \(i=10\), \([3]\) for \(i=11\), \([9]\) for
\(i=13\), \([2]\) for \(i=16\)).

A reader can now reconstruct our equivalence classes by reading
\(\vec{s}\) from left to right, given that additional information. (Try
it!) We should now count the number of dots \(\cdot\), since that equals
the number of times we have to give additional information. For a class
\([i']\) that is a vertex of degree \(\leq 2\), it can happen at most
once (that is, for at most one element \(i'\) of \([i']\)) that
\(s_i\ne 0,1,2\) for the index \(i\) in \(\mathbf{n}\) right after
\(i'\), unless \(1\in [i']\), in which case it can happen twice.
(Someone who already has a neighbor and will end up with \(\leq 2\)
neighbors in total can meet a new neighbor at most once.) For \([i']\) a
vertex of arbitrary degree, it can happen at most \(\kappa\) times that
\(s_i\ne 0\). Hence, writing \(n_{\leq 2}\) for the number of vertices
of degree \(\leq 2\) and \(n_{\geq 3}\) for the number of vertices of
degree \(\geq 3\), we see that the total number of indices
\(i\in \mathbf{n}\) with \(s_i\in \{*,.\}\) is at most
\(\kappa n_{\geq 3}+ n_{\leq 2} + 1 + 1\), where the last \(+1\) comes
from the first index \(i\) in \(\mathbf{n}\). The number of indices
\(i\) with \(s_i=*\) equals the number of classes, i.e.,
\(n_{\leq 2} + n_{\geq 3}\). Hence, the number of indices \(i\) with
\(s_i = .\) is
\[\leq \kappa n_{\geq 3}+ n_{\leq 2} + 2 - (n_{\geq 3} + n_{\leq 2}) = (\kappa-1) n_{\geq 3} + 2 \leq (\kappa-1) \nu + 2.
\]

Each equivalence class contributes a factor of at most
\(\mathscr{L} = \sum_{p\in \mathbf{P}} \frac{1}{p}\) to our total in
\eqref{eq:littlestar}; singletons (equivalence classes with
one element each) actually contribute \(\sqrt{\mathscr{L}}\), because of
the factor of \(\mathscr{L}^{-\frac{|\mathcal{S}(\sim)|}{2}}\). Recall
that we are saving a factor of almost
\(H_0^{\frac{\nu}{12 \kappa} - 1}\) through \eqref{eq:circ} (let us say
\(H_0^{\nu/24 \kappa}\), to be safe). Thus, forgetting for a moment
about the yellow equivalence classes, we conclude that the contribution
to \eqref{eq:circ} of the equivalence relations \(\sim\) such that
\(G_\sim\) has \(\nu\) vertices of degree \(\geq 3\) is
\[\ll 4^{2 k} 5^{|\mathbf{n}|} (2 k)^2 \left(\frac{(2 k)^{\kappa-1}}{H_0^{1/24 \kappa}}\right)^\nu \mathscr{L}^k,\]
where the factor of \(4^{2 k}\) is there because we also have to specify
\(\vec{\sigma}\in \{-1,1\}^{\{1,\dotsc,2 k\}}\) and
\(\mathbf{l}, \mathbf{n}\subset \{1,2,\dotsc, 2k\}\). Provided that we
set our parameters so that \(H_0^{1/24 \kappa}\geq 2 (2 k)^{\kappa-1}\)
(and it turns out that we may do so, provided that \(\log H_0\) is
larger than \((\log H)^{2/3+\epsilon}\) -- or rather, larger than
\((\log H)^{1/2+\epsilon}\), if we make the improvement through multiple
arrows we mentioned a little while ago), we are done; we have a bound of
size \[\ll \mathscr{L}^k \sum_{\nu=1}^\infty 2^{-\nu}
\ll \mathscr{L}^k,\] which is what we wanted all along.

But wait! What about the part of the word that disappears during
reduction? It is partly described by a string of matched parentheses:
for example, \(x x^{-1} x^{-1} y y^{-1} x\) gives us \(()(())\). (We
also have to specify the exponents \(\sigma_i\) separately.) The
equivalence class of the index of a closing parenthesis is the same as
that of the index of the matching opening parenthesis. Thus, we need
only worry about specifying the equivalence classes of the opening
parentheses. There are \(k-|\mathbf{n}|/2\) of them.

A naive approach would be to describe each such equivalence class
\([i]\) by specifying the first index \(i\) in it each time it occurs
(except for the first time). The cost of that approach could be about as
large as \(k^{k-|\mathbf{n}|/2}\), which is much too large. It would
seem we are in a pickle. Indeed, we know we would have to be in a
pickle, if we were not using the fact that we are not working in all of
\(\mathbf{N}\), but in a subset \(X\subset \mathbf{N}\) all of whose
elements have \(\leq K \mathscr{L}\) divisors in \(\mathbf{P}\). (If we
worked in all of \(\mathbf{N}\), even trivial walks, which are entirely
yellow, would pose an insurmountable problem.) However, how can we use
\(X\), or the bound \(\leq K \mathscr{L}\), by this point?

The point is that we need not consider all possible \((p_{[i]})\) in
\eqref{eq:littlestar}, but only those tuples that can possibly arise in a walk
\[n, n+\sigma_1 p_1, n+\sigma_1 p_1 + \sigma_2 p_2,\dotsc,
n+\sigma_1 p_1 + \sigma_2 p_2 + \dotsb + \sigma_{2 k} p_{2 k}=n
\] all of whose nodes are in \(X\). Now, if a prime \(p_j\) has appeared
before as \(p_i\) (i.e., \(i<j\) and \(i\sim j\)) and both \(i\) and
\(j\) are ``lit'', that is \(i,j\in \mathbf{l}\), then, as we know,
\(\sigma_{i} p_{i} + \dotsc + \sigma_{j-1} p_{j-1}\) must be divisible
by \(p_i\). (Indices that are not lit do not pose a problem, due to the
factors of the form \(1/p\) that they contribute.) What is more: if
\(i\in \mathbf{l}\), \(i<j\) with
\(p_i|\sigma_{i} p_{i} + \dotsc + \sigma_{j-1} p_{j-1}\), then
\(n + \sigma_1 p_1 + \dotsc + \sigma_{j-1} p_{j-1}\) is \emph{forced} to
be divisible by \(p_i\) (because
\(n+\sigma_1 p_1 + \dotsc + \sigma_{i-1} p_{i-1}\) is divisible by
\(p_i\)). Now, \(n+\sigma_1 p_1 + \dotsc + \sigma_{j-1} p_{j-1}\) has
\(\leq K \mathscr{L}\) divisors. Hence, given \(j\), there are at most
\(K \mathscr{L}\) distinct equivalence classes \([i]\) having at least
one representative \(i<j\), \(i\in \mathbf{l}\) such that
\(p_i|\sigma_{i} p_{i} + \dotsc + \sigma_{j-1} p_{j-1}\) . This is a
property where \(n\) no longer appears.

Now, as we describe \(\sim\) to our reader, when we come to an index of
the one kind that remains problematic -- disappearing in the reduction,
corresponding to an open parenthesis, in an equivalence class that has
been seen before -- we need only specify an equivalence class
\emph{among those \(\leq K \mathscr{L}\) equivalence classes that have
at least one representative \(i<j\), \(i\in \mathbf{l}\) such that}
\(p_i|\sigma_{i} p_{i} + \dotsc + \sigma_{j-1} p_{j-1}\). The reader can
figure out which one those are, as that is a property given solely by
\(p_1,\dotsc,p_{j-1}\) and \(\sigma_1,\dotsc,\sigma_{j-1}\). We can give
them numbers \(1\) to \(\lfloor K \mathscr{L}\rfloor\) by order of first
appearance, and communicate to the reader the equivalence class we want
by its number, rather than by an index. Thus we incur only in a factor
of \(K\mathscr{L}\), not \(2 k\).

In the end, we obtain a total contribution of \[O((K\mathscr{L})^k),\]
which is what we wanted. In other words,
\[\textrm{Tr} (A|_X)^{2 k} \leq O(K \mathscr{L})^{k} N,\] Q.E.D.

Incidentally, in earlier drafts of the paper, we did not have a
``writer'' and a ``reader'', but a mahout and an elephant:

\begin{figure}
\centering
\includegraphics[scale=0.45]{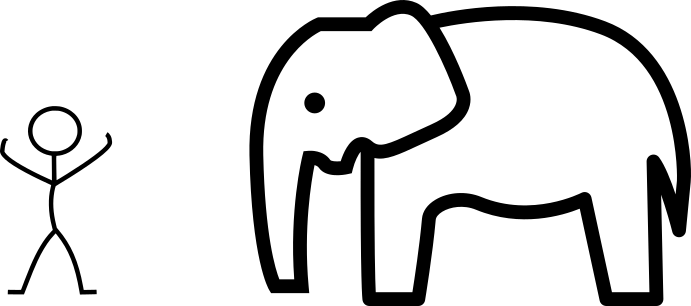}
\end{figure}

They were unfortunately censored by my coauthor. As this is my
exposition, here they are. The picture might be clearer now -- the
elephant-reader has no idea of \(n\), or of our grand strategy, but it
is an intelligent animal that can follow instructions and is endowed
with a flawless memory (and the ability to test for divisibility,
apparently).

\section{Conclusions}

\begin{main}Let the operator \(A\) be as before, with
\(\mathbf{N}=\{N+1,\dotsc,2 N\}\) and \(H_0,H,N\geq 1\) such that
\(H_0\leq H\) and \(\log H_0 \geq (\log H)^{1/2} (\log \log H)^{2}\).
Let \(\mathbf{P}\subset [H_0,H]\) be a set of primes such that
\(\mathscr{L} = \sum_{p\in \mathbf{P}} 1/p \geq e\) and
\(\log H \leq \sqrt{\frac{\log N}{\mathscr{L}}}\).

Then, for any \(1\leq K\leq \frac{\log N}{\mathscr{L} (\log H)^2}\),
there is a subset \(\mathscr{X}\subset \mathbf{N}\) with
\(|\mathbf{N}\setminus \mathscr{X}|\ll N e^{-K \mathscr{L} \log K} + N/\sqrt{H_0}\)
such that every eigenvalue of \(A|_{\mathscr{X}}\) is
\[O\left(\sqrt{K \mathscr{L}}\right),\] where the implied constants are
absolute. \end{main}

We have sketched a full proof, leaving out one, or rather two, passages
-- namely, the proof that we can take out from \(X\) two kinds of
integers, and still keep \(X\) well-distributed enough in arithmetic
progressions for cancellation to happen when we have too many lone
primes. As we have said before, those two kinds of integers are: (a)
integers \(n\) with \(\geq K \mathscr{L}\) divisors, (b) integers \(n\)
that could give rise to too many disjoint revenants. Here (b) sounds a
little vague, but, if we simply take out from \(X\) the set \(Y_\ell\)
of those integers \(n\) for which there can be a ``premature revenant'',
meaning that there exist \(p\in \mathbf{P}\),
\(p_1,\dotsc,p_l\in \mathbf{P}\) with \(p_i\ne p\) and
\(\sigma\in \{-1,1\}^l\), \(l\leq \ell\), such that
\[p|n, p_{1}|n, p_{2} | n + \sigma_1 p_1, \dotsc, p_l|n+\sigma_1 p_1 + \dotsc + \sigma_{l-1} p_{l-1}, p|n+\sigma_1 p_1 + \dotsc \sigma_l p_l,\]
then we have ensured that there cannot be more than \(2k/\ell\) disjoint
revenants. (We have not really forgotten about the possibility that some
intermediary indices may not be lit -- those are taken care of by a
different argument.) It is actually not hard to show that \(Y_{\ell}\)
is a fairly small set; what takes work is showing that it is
well-distributed. What we did was develop a new tool -- a combinatorial
sieve for conditions involving composite moduli. While it is somewhat
technical, may be interesting in that it will probably be useful for
attacking other problems. Let us leave it to the appendix.

The main theorem has several immediate corollaries. First of all, we
obtain what we set as our original goal.

\begin{cor} For any \(e<w\leq x\) such that \(w\to\infty\) as
\(x\to \infty\),
\[\frac{1}{\log w} \sum_{\frac{x}{w}\leq n\leq x} \frac{\lambda(n) \lambda(n+1)}{n} = O\left(\frac{1}{\sqrt{\log \log w}}\right).\]
\end{cor}

We can also obtain substantially sharper results. A case in point: we can
prove that \(\lambda(n+1)\) averages to zero (with weight \(1/n\) as
above, or ``at almost all scales'') over integers \(\leq N\) having
exactly \(k\) prime factors, where \(k\) is a popular number of prime
factors to have (e.g., \(\lfloor \log \log N\rfloor\), or
\(\lfloor \log \log N\rfloor + 2021\)). To see more such corollaries,
look at the actual paper, or derive your own!

\subsection{Subset of acknowledgments. Bonus track}

I am grateful to many people -- please read the full acknowledgments in
the paper. Here I would like to thank two subsets in particular -- (a)
postdocs and students in Göttingen who patiently attended my online
lectures during the first year of the COVID pandemic, as the proof was
finally gelling, (b) inhabitants of MathOverflow. In (b), one can find,
for example, Fedor Petrov, who pointed us towards Kleitman-West,
besides answering other questions, but you can also find some users who
chose to remain anonymous. Among them was user ``BS.'', who explained
how one of my question about ranks was related to topology. That
relation has gone well under the surface in the current version, so let
us discuss it here, for our own edification.

Consider a word \(w\) of a special kind -- a word \(w\) where every
letter \(x_1,\dotsc,x_k\) appears twice, once as \(x_i\), once as
\(x_i^{-1}\). For \(1\leq i,j\leq k\), let \(m_{i,j}\) equal \(1\) if
either (a) \(x_i\) appears before \(x_i^{-1}\) , and \(x_j\) appears
between them, but \(x_j^{-1}\) does not appear between them, or (b)
\(x_i^{-1}\) appears before \(x_i\), and \(x_j^{-1}\) appears between
them, but \(x_j\) does not. Let \(m_{i,j}=-1\) if either (a) or (b) is
true with \(x_j\) and \(x_j^{-1}\) switched. Let \(m_{i,j}=0\)
otherwise. Then the \(k\)-by-\(k\) matrix \(M=(m_{i,j})\) is
skew-symmetric. As people in MathOverflow kindly showed me (apparently
my education in linear algebra left something to be desired\ldots{}), if
a skew-symmetric matrix \(M\) has rank \(r\), then it has a minor with
disjoint row and column index sets and rank \(\geq r/2\). Since I was
interested precisely in constructing such a minor with high rank (\(I\)
and \(J\) giving us what we called ``blue'' and ``red'' vertices in the
above), it made sense that I would want to know what the rank \(r\) of
\(M\) might be. In particular, when is \(M\) non-singular?

What BS. showed to me is that one can construct a surface \(S\) with
handles corresponding to the word \(w\) in a natural way. (Apparently
this construction is standard, but it was completely unknown to me.) For
instance, for \(w = x_1 x_2 x_1^{-1} x_2^{-1} x_3 x_3^{-1}\), the
surface \(S\) looks as follows:

\begin{center}
\begin{tikzpicture}[scale=0.75]
 \colorlet{lightblue}{blue!20!white}
 \draw[thick, fill=lightblue] (2,-1) -- (2,3.2) .. controls (2,4) and (2.2,4.8) .. (3,5) .. controls (3.8,5.2) and (5.2,5.2) .. (6,5) .. controls (6.8,4.8) and (7,4) .. (7,3.2) -- (7,-1) -- (6,-1) -- (6,2) .. controls (6,3.2) and (5.8,3.8) .. (5.4,4) .. controls (5,4.2) and (4,4.2) .. (3.6,4) .. controls (3.2,3.8) and (3,3.2) .. (3,0);
 \draw[thick, fill=lightblue] (8,-1) -- (8,0) .. controls (8,2) and (8.2,2.8) .. (9,3) .. controls (9.5,3.05) .. (10,3) .. controls (10.8,2.8) and (11,2) .. (11,0) -- (11,-1) -- (10,-1) -- (10,0) .. controls (10,1.2) and (9.9,1.8) .. (9.7,2) .. controls (9.55,2.1) and (9.45,2.1) ..(9.3,2) .. controls (9.1,1.8) and (9,1.2) .. (9,0) -- (9,-0.5) -- (8,-1);
 \draw[thick, fill=lightblue] (0,-1) -- (0,0) .. controls (0,2) and (0.2,2.8) .. (1,3) .. controls (2,3.2) and (3,3.2) .. (4,3) .. controls (4.8,2.8) and (5,2) .. (5,0) -- (5,-1) -- (4,-1) -- (4,-0.5) -- (4,0) .. controls (4,1.2) and (3.8,1.8) .. (3.4,2) .. controls (3,2.2) and (2,2.2) .. (1.6,2) .. controls (1.2,1.8) and (1,1.2) .. (1,0);
 \draw[fill=lightblue] (0,0) -- (0,-1) .. controls (0,-1.75) and (0.25,-2) .. (1,-2) -- (10,-2) .. controls (10.75,-2) and (11,-1.75) .. (11,-1) -- (11,0); 
 \draw[thick] (1,0) -- (2,0);
 \draw[thick] (3,0) -- (4,0);
 \draw[thick] (5,0) -- (6,0);
 \draw[thick] (7,0) -- (8,0);
 \draw[thick] (9,0) -- (10,0);
\end{tikzpicture}
\end{center}

The matrix \(M\) then corresponds to the intersection form of this
surface. This form is defined as an antisymmetric inner product on
\(H_1(S,\mathbb{Z})\), counting the number of intersections (with
orientation) of two closed paths in the way you may expect. For
instance, in the following, \(\langle z_1,z_2\rangle=-1\), whereas
\(\langle z_1,z_3\rangle = \langle z_2,z_3\rangle = 0\):

\begin{center}
\begin{tikzpicture}[scale=0.75]
   \colorlet{lightblue}{blue!20!white}
  \draw[thick, fill=lightblue] (2,-1) -- (2,3.2) .. controls (2,4) and (2.2,4.8) .. (3,5) .. controls (3.8,5.2) and (5.2,5.2) .. (6,5) .. controls (6.8,4.8) and (7,4) .. (7,3.2) -- (7,-1) -- (6,-1) -- (6,2) .. controls (6,3.2) and (5.8,3.8) .. (5.4,4) .. controls (5,4.2) and (4,4.2) .. (3.6,4) .. controls (3.2,3.8) and (3,3.2) .. (3,0);
  \draw[thick,->,color=violet] (6.5,0.5) -- (6.5,2.5);
  \draw[thick,->,color=violet] (6.5,2.5) .. controls (6.5,3.7) and (6.3,4.3) .. (5.9,4.5);
  \draw[thick,->,color=violet] (5.9,4.5) .. controls (5.5,4.7) and (3.5,4.7) .. (3.1,4.5);
  \draw[thick,color=violet] (3.1,4.5) .. controls (2.7,4.3) and (2.5,3.7) .. (2.5,2.5);
  \draw[thick,color=violet] (2.5,2.5) -- (2.5,0.5);
  \draw[thick, fill=lightblue] (8,-1) -- (8,0) .. controls (8,2) and (8.2,2.8) .. (9,3) .. controls (9.5,3.05) .. (10,3) .. controls (10.8,2.8) and (11,2) .. (11,0) -- (11,-1) -- (10,-1) -- (10,0) .. controls (10,1.2) and (9.9,1.8) .. (9.7,2) .. controls (9.55,2.1) and (9.45,2.1) ..(9.3,2) .. controls (9.1,1.8) and (9,1.2) .. (9,0) -- (9,-0.5) -- (8,-1);
  \draw[thick, fill=lightblue] (0,-1) -- (0,0) .. controls (0,2) and (0.2,2.8) .. (1,3) .. controls (2,3.2) and (3,3.2) .. (4,3) .. controls (4.8,2.8) and (5,2) .. (5,0) -- (5,-1) -- (4,-1) -- (4,-0.5) -- (4,0) .. controls (4,1.2) and (3.8,1.8) .. (3.4,2) .. controls (3,2.2) and (2,2.2) .. (1.6,2) .. controls (1.2,1.8) and (1,1.2) .. (1,0);
  \draw[fill=lightblue] (0,0) -- (0,-1) .. controls (0,-1.75) and (0.25,-2) .. (1,-2) -- (10,-2) .. controls (10.75,-2) and (11,-1.75) .. (11,-1) -- (11,0);  
  \draw[thick] (1,0) -- (2,0);
  \draw[thick] (3,0) -- (4,0);
  \draw[thick] (5,0) -- (6,0);
  \draw[thick] (7,0) -- (8,0);
  \draw[thick] (9,0) -- (10,0);
  \draw[thick,->,color=violet] (0.9,-0.5) -- (3,-0.5);
  \draw[thick,color=violet] (3,-0.5) -- (4.1,-0.5) .. controls (4.5,-0.4) .. (4.5,0.5);
  \draw[thick,color=violet] (0.9,-0.5) .. controls (0.5,-0.4) .. (0.5,0.5);
  \draw[thick,->,color=violet] (4.5,0.5) .. controls (4.5,1.7) and (4.3,2.3) .. (3.9,2.5) node[below] {$z_1$};
  \draw[thick,->,color=violet] (3.9,2.5) .. controls (3.5,2.7) and (1.5,2.7) .. (1.1,2.5);
  \draw[thick,color=violet] (1.1,2.5) .. controls (0.7,2.3) and (0.5,1.7) .. (0.5,0.5);
  \draw[thick,->,color=violet] (2.9,-1) -- (5,-1);
  \draw[thick,color=violet] (5,-1) node[below] {$z_2$} -- (6.1,-1) .. controls (6.5,-0.9) .. (6.5,0.5);
  \draw[thick,color=violet] (2.9,-1) .. controls (2.5,-0.9) .. (2.5,1);
  \draw[thick,->,color=violet] (8.9,-0.5) -- (9,-0.5) node[below] {$z_3$};
  \draw[thick,color=violet] (9,-0.5) -- (10.1,-0.5) .. controls (10.5,-0.4) .. (10.5,0.5);
  \draw[thick,color=violet] (8.9,-0.5) .. controls (8.5,-0.4) .. (8.5,0.5);
  \draw[thick,->,color=violet] (10.5,0.5) .. controls (10.5,1.7) and (10.3,2.3) .. (9.9,2.5);
  \draw[thick,color=violet] (9.9,2.5) .. controls (9.5,2.6) .. (9.1,2.5);
  \draw[thick,color=violet] (9.1,2.5) .. controls (8.7,2.3) and (8.5,1.7) .. (8.5,0.5);
  \end{tikzpicture}
\end{center}

Say \(S\) has genus \(g\) and \(b\geq 1\) boundary components. Then, for
\(S_g\) the surface of genus \(g\) without boundary, there is an
embedding \(S\hookrightarrow S_g\) preserving the intersection form,
with \(H^1(S)\to H^1(S_g)\) having kernel of rank \(b-1\). The
intersection form on \(H^1(S_g)\) is non-singular. Hence, \(M\) has
corank \(b-1\). In particular, \(M\) is non-singular iff \(b=1\), i.e.,
iff its boundary is connected.

It is an exercise to show that \(b\) equals the number of cycles in the
permutation \(i\mapsto \sigma(i)+1 \bmod 2 k\), where \(\sigma\) is the
permutation of \(\{1,2,\dotsc,2 k\}\) switching \(x_i\) and \(x_i^{-1}\)
in \(w\) for every \(1\leq i\leq k\).

I have no idea of how to define a surface \(S\) like the above for a
word \(w\) of general form -- the natural generalization of \(M\) is the
matrix corresponding to the system 
\eqref{eq:ddagger} of divisibility relations, and that matrix need not be skew-symmetric, or even
square. However, in \(S\) and its boundary, you can already see shades
of our graph \(\mathscr{G}_{(\sigma,\sim)}\).

\appendix
\section{Sieves}

What we must address now may be seen as a technical task. However, the
way we will address it most likely has more general applicability.

Our task in this appendix is to show how to exclude from our set
\(X\subset \mathbf{N}\) all integers \(n\) that could give rise to
\emph{premature revenants}, that is, edge lengths \(p_i=p_i'\) with
\(i'-i\) small such that \(p_j\ne p_i\) for some \(i<j<i'\). (Without
this last condition, we would be counting not only ``revenants'' but
also mere repetitions.) As we already commented, it is enough to exclude
the set \(Y_\ell\) of all integers \(n\) such that there exist
\(p\in \mathbf{P}\), \(p_1,\dotsc,p_l\in \mathbf{P}\) with \(p_i\ne p\)
and \(\sigma\in \{-1,1\}^l\), \(l\leq \ell\), for which
\begin{equation}\label{eq:dagger}\begin{aligned}
p|n, p_{1}&|n, p_{2} | n + \sigma_1 p_1, \dotsc, \\p_l&|n+\sigma_1 p_1 + \dotsc + \sigma_{l-1} p_{l-1}, p|n+\sigma_1 p_1 + \dotsc \sigma_l p_l.\end{aligned}\end{equation}

It is actually easy (and in fact an exercise for the reader) to show
that \(Y_{\ell}\) is quite small -- not much larger than
\(O(\mathscr{L})^{\ell} N/H_0\). (Outline: the probability that a given
divisor \(p\) of \(\sigma_1 p_1 + \dotsc + \sigma_l p_l\) divide a
random \(n\) is about \(1/p\), which is at most \(1/H_0\); the
probability that there be \(p_1,\dotsc,p_l\) as above with
\(\sigma_1 p_1 + \dotsc + \sigma_l p_l=0\) is also quite small.) The more
complicated task is to show that \(Y_{\ell}\) is reasonably
equidistributed in arithmetic progressions.

The main issue here is that we have a great number of conditions as in \eqref{eq:dagger}
to exclude. Inclusion-exclusion involves \(2^m\) terms for
\(m\) conditions -- that is too many. There is a tool for dealing with
that sort of issue in number theory, at least in some specific contexts:
sieves.

We shall first show how to set up a general, abstract combinatorial
sieve, for arbitrary logical conditions (rather than conditions of the
form \(n\equiv a \bmod p\)). We will then show how to apply it to
conditions of the form \(n\equiv a \bmod m\), that is, congruence
conditions where the moduli may be composite (as opposed to being prime,
as is common in sieve theory). The matter is tricky -- one has to prevent
combinatorial explosion again. Rota's cross-cut theorem will be our
friend.

Lastly, we will show how to apply the sieve in our context (with moduli
\(p p_1 \dotsb p_l\) coming from \eqref{eq:dagger} and sketch how to
estimate the main and error terms. We will introduce \emph{sieve
graphs}.

For readers who have had some passing contact with sieve theory: while,
in introductory texts on sieve theory, the emphasis is often on
\emph{counting} a set of elements \(S\) not obeying any of a set of
conditions (e.g., the set \(S\) of primes \(n\) such that \(n+2\) is
also a prime; its elements \(n\) do not fulfill the conditions
\(n\equiv 0 \bmod p\) or \(n\equiv -2 \bmod p\) for any small prime
\(p\)), the emphasis on much recent work, and also here, lies more
generally on providing an approximation to the characteristic function
\(1_S\) of \(S\) by a function that is easier to deal with, or, if you
wish, has a ``simpler description'' (in some precise sense). One label
that has become attached to this use of sieves is ``enveloping sieve'',
though that really describes one kind of approximation (a majorant of
\(1_S\)) and at any rate should really be called an enveloping
\emph{use} of a sieve (many sieves can be used as enveloping sieves). At
any rate, that is all more or less orthogonal to the main issue here,
which is that we have to develop a genuinely more general sieve.

\subsection{An abstract combinatorial sieve}

Let \(\mathbf{Q}\) be a set of conditions that an element \(x\) of a set
\(Z\) may or may not fulfill. (For us, later, \(Z\) will be the set of
integers, but that is of no importance at this point.) Denote by
\(\mathbf{Q}(x)\in 2^{\mathbf{Q}}\) the set
\(\{Q\in \mathbf{Q}: Q(x) \text{ is true}\}\), i.e., the set of
conditions in \(\mathbf{Q}\) fulfilled by \(x\). Define
\(1_\emptyset(S)\) to be \(1\) if the set \(S\) is empty, and \(0\)
otherwise. Then \(1_\emptyset(\mathbf{Q}(x))\) equals \(1\) when \(x\)
satisfies none of the conditions in \(\mathbf{Q}\), and \(0\) otherwise.

We are interested in approximations to \(1_\emptyset(\mathbf{Q}(x))\),
i.e., the function that takes the value \(1\) when \(x\) satisfies none
of the conditions in \(\mathbf{Q}\), and \(0\) otherwise. This may seem
to be a silly question, though it falls within the general framework we were
discussing before. Let us put matters a little differently. A standard
way to express \(1_\emptyset(\mathbf{Q}(x))\) would be as
\[1_\emptyset(\mathbf{Q}(x)) = \sum_{\mathbf{T}\subset \mathbf{Q}(x)} (-1)^{|\mathbf{T}|},\]
and that might suit us, except that the number of subsets
\(\mathbf{T}\subset \mathbf{Q}(x)\) is very large. Can we obtain a
reasonable approximation by means of a sum of the form
\[\sum_{\mathbf{T}\subset \mathbf{Q}(x)} g(\mathbf{T}) (-1)^{|\mathbf{T}|},\]
where \(g:2^\mathbf{Q}\to \{0,1\}\) is a function -- preferably one
whose support is much smaller than \(\mathbf{Q}(x)\)? (Here, as is
usual, \(2^\mathbf{Q}\) denotes the set of all subsets of
\(\mathbf{Q}\).)

It turns out to be possible to bound the error term in an approximation
of this form in full generality. To be precise: the error term will be
bounded in terms of the boundary of the support of \(g\). Here we say
that an \(\mathbf{S}\subset \mathbf{Q}\) is in the boundary of a
collection \(B\subset 2^\mathbf{Q}\) if there is an element \(s\) of
\(\mathbf{S}\) such that exactly one of the two sets \(\mathbf{S}\),
\(\mathbf{S}\setminus \{s\}\) is in \(B\).

\begin{lemma}
Let \(g:2^\mathbf{Q}\to \{0,1\}\). Assume
\(g(\emptyset)=1\). Choose a linear ordering for \(\mathbf{Q}\). Then
\[\begin{aligned}
1_\emptyset(\mathbf{Q}(x)) &= 
\sum_{\mathbf{T}\subset \mathbf{Q}(x)} g(\mathbf{T}) (-1)^{|\mathbf{T}|}\\
&+ \mathop{\sum_{\emptyset\ne \mathbf{S}\subset \mathbf{Q}(x)}}_{\min(\mathbf{Q}(x))\in \mathbf{S}} (-1)^{|\mathbf{S}|}
(g(\mathbf{S}\setminus \{\min(\mathbf{Q}(x))\})-g(\mathbf{S})).
\end{aligned}
\]
\end{lemma}

The proof is short and basically trivial (for \(\mathbf{Q}(x)\)
non-empty, the second sum is just a reordering of the first sum, with
opposite sign). It is inspired by a passage in the proof of Brun's
combinatorial sieve (see, e.g., \cite[\S 6.2, p.~87-89]{zbMATH02239783}.
We do not need the linear ordering for \(\mathbf{Q}\) to be in any sense
natural.

\subsection{Sieving by composite moduli}

Let \(\mathscr{Q}\) be a finite collection of arithmetic progressions.
To each progression \(P\in \mathscr{Q}\), we can associate the condition
\(n\in P\), for \(n\in \mathbb{Z}\). Thus, we obtain a set
\(\mathbf{Q}\) of conditions corresponding to \(\mathscr{Q}\), and apply
the framework above.

We are interested in approximating
\(1_{n\in P \forall P\in \mathscr{Q}}(n)\) -- that is, the
characteristic function of the set of all \(n\) lying in no progression
\(P\in \mathscr{Q}\) -- by a sum
\[F_\mathfrak{D}(n) = \mathop{\sum_{\mathscr{S}\subset \mathscr{Q}}}_{\bigcap \mathscr{S} \in \mathfrak{D}} (-1)^{|\mathscr{S}|} 1_{n\in \bigcap \mathscr{S}},\]
where \(\mathfrak{D}\subset \mathscr{Q}^\cap\) is some set of
progressions.

We will denote by \(\mathfrak{q}(R)\) the modulus \(q\) of an
arithmetic progression \(a + q \mathbb{Z}\).

\begin{prop}Let \(\mathscr{Q}\) be a finite collection of
distinct arithmetic progressions in \(\mathbb{Z}\) with square-free
moduli. Let \(\mathfrak{D}\) be a non-empty subset of
\(\mathscr{Q}^\cap = \{\bigcap \mathscr{S}: \mathscr{S}\subset \mathscr{Q}\}\)
with \(\emptyset\not\in \mathfrak{D}\). Assume \(\mathfrak{D}\) is
closed under containment, i.e., if \(S\in \mathfrak{D}\), then every
superset \(S'\supset S\) in \(\mathscr{Q}^\cap\) is also in
\(\mathfrak{D}\). Let \(F_\mathfrak{D}\) be as above. Then
\[\begin{aligned} 1_{n\in P \forall P\in \mathscr{Q}}(n) &=
F_{\mathfrak{D}}(n) +
O^*\left(\sum_{R\in \partial \mathfrak{D}} 2^{\omega(\mathbf{q}(R))} 1_{n\in R}\right)\\
&= F_{\mathfrak{D}}(n) +
O^*\left(\sum_{R\in \partial_{\mathrm{out}} \mathfrak{D}} 3^{\omega(\mathbf{q}(R))} 1_{n\in R}\right)
\end{aligned}\] where \[\partial \mathfrak{D} = \{R\in \mathfrak{D}:
\exists P\in \mathscr{Q}\; \text{s.t.}\; P\cap R \not\in \mathfrak{D}\},\]
\[\partial_\mathrm{out} \mathfrak{D} = \{D\in \mathscr{Q}^\cap \setminus
\mathfrak{D} : \exists P\in \mathscr{Q}, R \in \mathfrak{D}\, \text{s.t.}\, D = P \cap R \}.\]
Moreover, we can write \(F_\mathfrak{D}(n)\) in the form
\[F_\mathfrak{D}(n) = \sum_{R\in \mathfrak{D}} c_R 1_{n\in R}\] with
\(|c_R|\leq 2^{\omega(\mathfrak{q}(R))}\).
\end{prop}
We can of course think of \(\partial \mathfrak{D}\) and
\(\partial_{\mathrm{out}} \mathfrak{D}\) as the boundary and the outer
boundary of \(\mathfrak{D}\).

\begin{proof}The proof of the Proposition starts with an
application of the Lemma above. In what then follows, the important
thing is to prevent a combinatorial explosion. For instance, it is not a
priori clear that \(c_R\) can be bounded well: there could be very many
ways to express a given \(R\in \mathfrak{D}\) as an intersection
\(\bigcap \mathscr{S}\); in fact, the number of ways could be close to
\(2^{2^{\omega(\mathfrak{q}(R))}}\), that is, the number of collections
of subsets of a set with \(\mathfrak{q}(R)\) elements. We can give the
much better bound \(2^{\omega(\mathfrak{q}(R))}\) by obtaining
cancellation (by \((-1)^{|\mathscr{S}|}\)) among those different ways.
To be more precise, we apply the following Lemma, which is an easy
consequence of Rota's cross-cut theorem, but can also be proved from
scratch in a couple of lines. The same Lemma allows us to deal with the
same combinatorial explosion in the error terms.
\end{proof}

\begin{lemma} Let \(\mathscr{C}\) be a collection of subsets of a
finite set \(X\). Then
\[\left|\mathop{\sum_{\mathscr{S}\subset \mathscr{C}}}_{\bigcup \mathscr{S} = X} (-1)^{|\mathscr{S}|}\right|\leq 2^{|X|}.\]
\end{lemma}
\begin{proof}
 Exercise.
\end{proof}

How to apply the Proposition? We can define \(\mathfrak{D}\) to be the
set of progressions in \(\mathscr{Q}^\cap\) with ``small modulus'', for
some notion of ``small''. Then its boundary consists of progressions
that are ``borderline small'', i.e., not really small, and so the
proportion of \(n\) in each one of them will not be large; we just need
to control the size of the boundary to show that the total error term is
acceptable.

\subsection{Sieve graphs and their usage}

We now come to our application of the sieve we have just developed. Our
aim is to prevent our walks
\[n, n+\sigma_1 p_1, n+\sigma_1 p_1 + \sigma_2 p_2,\dotsc\] from having
what we called \emph{premature revenants}. We will do so by constraining
each of \(n, n+\sigma_1 p_1,n+\sigma_1 p_1 + \sigma_2 p_2\dotsc\) to lie
within the set \(Y_{\ell}\) of integers that \emph{cannot} give rise to
premature revenants.

To be precise: we define \(Y_{\ell}\) to be the set of all integers
\(n\) except for those for which there are primes
\(p_1,\dotsc,p_l\in \mathbf{P}\) and signs
\(\sigma_1,\dotsc,\sigma_l\in \{-1,1\}\) with \(1\leq l< \ell\) such
that
\begin{equation}\label{eq:etoile}p_1|n, p_2|n+\sigma_1 p_1,\dotsc,p_l|n+\sigma_1 p_1 + \dotsc + \sigma_{l-1} p_{l-1},\end{equation}
there are no repeated primes among \(p_1,\dotsc,p_l\) except perhaps for
consecutive primes \(p_i=p_{i+1} = \dotsc =p_j\) with
\(\sigma_i = \sigma_{i+1} = \dotsc = \sigma_j\), and one of the
following two conditions holds:
\begin{itemize}
\item there exists a prime
\(p_0\in \mathbf{P}\) distinct from \(p_1,\dotsc,p_l\) such that
\begin{equation}\label{eq:etoile1}p_0|n\;\text{and}\;p_0|n+\sigma_1 p_1 + \dotsc + \sigma_l p_l,\end{equation}
\item we have
\begin{equation}\label{eq:etoile2}
\sigma_1 p_1 + \dotsc + \sigma_l p_l= 0.\end{equation}
\end{itemize}

The set of integers \(n\) that obey conditions \eqref{eq:etoile}
and \eqref{eq:etoile1} is
an arithmetic progression to modulus \([p_0,p_1,\dotsc,p_l]\) (that is,
the lcm of \(p_1,\dotsc,p_l\), i.e., the product of all distinct primes
among them), unless it is empty. The set of integers \(n\) obeying
\eqref{eq:etoile} and \eqref{eq:etoile2} is an arithmetic progression to modulus
\([p_1,p_2,\dotsc,p_l]\), unless it is empty. Let
\(W_{\ell,\mathbf{P}}\) denote the set of all arithmetic progressions
arising in this way. Then the condition \(n\in Y_\ell\) is equivalent to
\(n\) not lying in any of the arithmetic progressions in
\(W_{\ell,\mathbf{P}}\). Likewise, for
\(\beta_1,\dotsc,\beta_{2 k}\in \mathbb{Z}\), the condition that
\(n+\beta_i\in Y_{\ell}\) for all \(1\leq i\leq 2 k\) is equivalent to
asking that \(n\) not be in any arithmetic progression of the form
\(P-\beta_i\) with \(P\in W_{\ell,\mathbf{P}}\) and \(1\leq i\leq 2 k\).

We are thus in the kind of situation to which our sieve for
composite moduli is applicable. Applying the Proposition above, we
obtain, for any \(m\geq 1\),
\[1_{n+\beta_i\in Y_{\ell} \forall 1\leq i\leq 2 k}=
\mathop{\sum_{R\in \mathscr{Q}^\cap}}_{\omega(\mathfrak{q}(R))\leq m} c_R 1_{n\in R} + O^*\left(3^{m+\ell} \mathop{\sum_{R\in \mathscr{Q}^\cap}}_{m<\omega(\mathfrak{q}(R))\leq m+\ell} 1_{n\in R} \right)
\] for \(\mathscr{Q} = W_{\ell,\mathbf{P}}(\beta)\) and some
\(c_R\in \mathbb{R}\) with \(|c_R|\leq 2^{\omega(\mathfrak{q}(R))}\).
(The proof is one line: we define \(\mathfrak{D}\) to be the set of all
non-empty \(R\in \mathscr{Q}^\cap\) such that the modulus of \(R\) has
\(\leq m\) prime factors.)

The more obvious issue now is how to bound the error term here. (For us, in
our application, there
is also the related issue of showing that the main term is well-behaved
-- in particular, its sum over certain arithmetic progressions should
not be too large.)

To keep track of the kind of conditions giving rise to a progression
\(R\in \mathscr{Q}^\cap\), we find it sensible to define a \emph{sieve
graph}, which is, one may say, a pictorial representation of those
conditions, or rather of what their general shape is and how they relate to
each other.

We define a \emph{sieve graph} to be a directed graph consisting of:
\begin{enumerate}
\item
a path of length \(2 k\), called the \emph{horizontal path};
\item
\emph{threads} of length \(< \ell\), of two kinds:
\begin{enumerate}
\item a \emph{closed
thread}, which is a cycle that contains some vertex of the horizontal
path, and is otherwise disjoint from it,
\item an \emph{open thread}, which
is a path that has an endpoint at some vertex of the horizontal path,
and is otherwise disjoint from it;
\end{enumerate}
\item for each open thread and each of
the two endpoints of that thread, an edge whose tail is that endpoint,
but whose head belongs only to the edge (i.e., it is a vertex of degree
\(1\)). These two edges will be called the thread's \emph{witnesses};
they are considered to be part of the thread.
\end{enumerate}

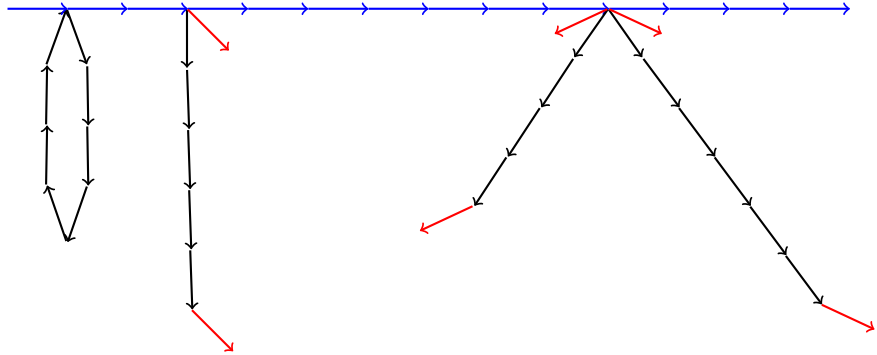
\begin{figure}
\begin{tikzpicture}[thick,scale=0.8,->,fill=black!50]
  \tikzstyle{vertex}=[circle,fill=black!50,minimum size=0pt, inner sep=0pt]
        
          \foreach \x in {0,1,...,14}
             \node[vertex] (G-\x) at (\x,0) {};

         \node[vertex] (C-1) at ([shift=(-70:1)]{G-1}) {};
         \draw (G-1) -> (C-1);
         \node[vertex] (C-2) at ([shift=(-90:1)]{C-1}) {};
         \node[vertex] (C-3) at ([shift=(-90:1)]{C-2}) {};
         \node[vertex] (C-4) at ([shift=(-110:1)]{C-3}) {};
         \node[vertex] (C-5) at ([shift=(110:1)]{C-4}) {};
         \node[vertex] (C-6) at ([shift=(90:1)]{C-5}) {};
         \node[vertex] (C-7) at ([shift=(90:1)]{C-6}) {};
         \draw (C-7) -> (G-1);
         \foreach \x in {1,...,6}
              {
                \tikzmath{\y=\x+1;}
                \draw (C-\x) -> (C-\y);
              };
              
         \foreach \x in {0,1,...,13}
              {
                \tikzmath{\y=\x+1;}
                \draw[blue] (G-\x) -> (G-\y);
              };

        \node[vertex] (O-1) at ([shift=(-90:1)]{G-3}) {};
        \draw (G-3) -> (O-1);
        \foreach \x in {2,...,5}
              {
                \tikzmath{\y=\x-1;}
                \node[vertex] (O-\x) at ([shift=(-90:1)]{O-\y}) {};
              }
        \foreach \x in {1,...,4}
              {
                \tikzmath{\y=\x+1;}
                \draw (O-\x) -> (O-\y);
              };
        \node[vertex] (W-1) at ([shift=(-45:1)]{G-3}) {};
        \node[vertex] (W-2) at ([shift=(-45:1)]{O-5}) {};
        \draw[red] (G-3) -> (W-1);
        \draw[red] (O-5) -> (W-2);

        \node[vertex] (OT-1) at ([shift=(-125:1)]{G-10}) {};
        \draw (G-10) -> (OT-1);
        \foreach \x in {2,...,4}
              {
                \tikzmath{\y=\x-1;}
                \node[vertex] (OT-\x) at ([shift=(-125:1)]{OT-\y}) {};
              }
        \foreach \x in {1,...,3}
              {
                \tikzmath{\y=\x+1;}
                \draw (OT-\x) -> (OT-\y);
              };
        \node[vertex] (WT-1) at ([shift=(-155:1)]{G-10}) {};
        \node[vertex] (WT-2) at ([shift=(-155:1)]{OT-4}) {};
        \draw[red] (G-10) -> (WT-1);
        \draw[red] (OT-4) -> (WT-2);

        \node[vertex] (OX-1) at ([shift=(-55:1)]{G-10}) {};
        \draw (G-10) -> (OX-1);
        \foreach \x in {2,...,6}
              {
                \tikzmath{\y=\x-1;}
                \node[vertex] (OX-\x) at ([shift=(-55:1)]{OX-\y}) {};
              }
        \foreach \x in {1,...,5}
              {
                \tikzmath{\y=\x+1;}
                \draw (OX-\x) -> (OX-\y);
              };
        \node[vertex] (WX-1) at ([shift=(-25:1)]{G-10}) {};
        \node[vertex] (WX-2) at ([shift=(-25:1)]{OX-6}) {};
        \draw[red] (G-10) -> (WX-1);
        \draw[red] (OX-6) -> (WX-2);
        
\end{tikzpicture}
\caption{A sieve graph: the blue path is the
horizontal path, and the witness edges of each open path are in red.
Any resemblance to a khipu is both intentional and ahistorical.}
\end{figure}

We will work with pairs \((G,\sim)\), where \(G\) is a sieve graph and
\(\sim\) is an equivalence relation on the edges of \(G\) such that the
witnesses of a thread are equivalent to each other. We can put
additional conditions, reflecting the conditions defining \(Y_{\ell}\):
we require that witnesses be equivalent to no other edges in their
thread, and that, in any thread, the set of edges in an equivalence
class form a connected subgraph (meaning: primes do not repeat in a thread unless
they are consecutive).

At any rate, it is clear what we will do: we will go over different
pairs \((G,\sim)\), and, for each pair, we will consider the
divisibility conditions resulting from assigning a distinct prime in
\(\mathbf{P}\) to each equivalence class of \(\sim\). We recall that
\(Y_{\ell}\) is defined as the set of integers that do not satisfy any
conditions of a certain kind. A pair \((G,\sim)\), together with an
assignment of a prime in \(\mathbf{P}\) to each equivalence class of
\(\sim\), corresponds to a conjunction
\(Q_1\wedge Q_2\wedge \dotsb \wedge Q_j\) of some such conditions
\(Q_i\). (To be precise - the conditions are given by a pair
\((G,\sim)\) together with a subset
\(\mathbf{l}\subset \{1,2,\dotsc,2k\}\), corresponding to the ``lit''
edges: only those edges in the horizontal path whose indices are in
\(\mathbf{l}\) impose divisibility conditions - the unlit edges are
muted, so to speak.) We need to study these conditions \(Q_i\) (all of
which are of the form ``\(n\) belongs to an arithmetic progression'', as
we have seen) because they will appear in the approximation to
\(1_{Y_{\ell}}\) that a sieve will give us.

We say \((G,\sim)\) is \emph{non-redundant} if every thread contains at
least one edge \(x\) (possibly a witness) whose equivalence class
\([x]\) contains no edge in any other thread. (A thread where every edge
is equivalent to an edge in some other thread would correspond to a
condition that either is redundant, given the conditions from the other
threads, or contradicts them. As Caliph Omar did \emph{not} say\ldots{})
The \emph{cost} \(\kappa(G,\sim)\) is the number of equivalence classes
that contain at least one edge (possibly a witness) in some thread,
i.e., the number of classes that do not contain only edges in the
horizontal path. Let \(\mathbf{W}_{k,\ell,m}\) be the set of
non-redundant pairs with given parameters \(k\) and \(\ell\) and cost
\(m\). It is clear that \(\mathbf{W}_{k,\ell,m}\) must be finite, since
any pair of cost \(m\) contains at most \(m\) threads. It is not hard to
bound the number of elements of \(\mathbf{W}_{k,\ell,m}\) with a given
number of threads \(r\leq m\).

When we apply our sieve for composite moduli so as to approximate
$1_{n+\beta_i\in Y_\ell \forall 1\leq i\leq 2k}$,
we define \(\mathfrak{D}\) to be the set of conditions
corresponding to non-redundant pairs \((G,\sim)\) with cost
\(\kappa(G,\sim)\leq m\) for some value of \(m\) we choose. Then the
outer boundary \(\partial_{\mathrm{out}} \mathfrak{D}\) corresponds to
non-redundant pairs \((G,\sim)\) of cost
\(m<\kappa(G,\sim)\leq m+\ell\). Our task is then to prove that the
contribution of all \((G,\sim)\) with cost between \(m\)
and \(m+\ell\) is small. The crucial part is to show that, given any
such \((G,\sim)\), together with a sign \(\sigma_y\) for each edge \(y\)
and a subset \(\mathbf{l}\subset \mathbf{k}\), the sum of
\[\prod_{i\in \mathbf{k}\setminus \mathbf{l}} \frac{1}{p_{[i]}}
\prod_{[x]\not\subset \mathbf{k}\setminus \mathbf{l}} \frac{1}{p_{[x]}}
\] over all choices of \(p_{[x]}\in \mathbf{P}\) (per equivalence class
\([x]\) of \(\sim\)) is small: it is at most
\[\mathscr{L}^{s-r} \left(\frac{\log H}{H_0}\right)^r,\] where \(r\) is
the number of threads in \(G\) and \(s\) is the number of equivalence
classes of \(\sim\). (Recall that \(\mathbf{P}\in [H_0,H]\) and
\(\mathscr{L} = \sum_{p\in \mathbf{P}} 1/p\).)

The proof is simple, and its main idea is as follows. Since \((G,\sim)\)
is non-redundant, each thread contains an edge in an equivalence class
that does not appear in other threads (or elsewhere in the same thread,
except for consecutive appearances). For each thread, we choose one such
edge \(x\). Then the thread binds the variable \(p_{[x]}\), so to speak,
and so we lose one degree of freedom for each \(r\). In detail:

\begin{enumerate}
\def\labelenumi{\arabic{enumi}.}
\tightlist
\item
  for a closed thread, the sum \(\sum_y \sigma_y p_{[y]}\) over the
  edges \(y\) of the thread is \(0\), and so \(p_{[x]}\) is determined
  by the other \(p_{[y]}\);
\item
  for an open thread and \(x\) a witness, \(p_{[x]}\) can range only
  over the prime divisors of \(\sum_y \sigma_y p_{[y]}\);
\item
  for an open thread and \(x\) not a witness, given \(p_{[y]}\) for the
  other \([y]\) in the thread, and given \(p_{[z]}\) for \(z\) the
  thread's witness, the class of \(p_{[x]}\) modulo \(p_{[z]}\) is
  determined.
\end{enumerate}

\begin{center}\rule{0.5\linewidth}{0.5000000000pt}\end{center}

We can give estimates on the main term of the sieve in much the same
way, only keeping track of the set \(\mathbf{W}_{k,\ell,m}'\) of
\emph{strongly non-redundant pairs}, meaning pairs \((G,\sim)\) such
that every thread contains at least one edge \(x\) (possibly a witness)
whose equivalence class \([x]\) contains no edge in any other thread and
no lit edge in the horizontal path (that is, no edge in the horizontal
path with index in \(\mathbf{l}\)).

\bibliographystyle{alpha}
\bibliography{trace}

\begin{thebibliography}{MRT15}

\bibitem[CM06]{zbMATH02239783}
A.~C. {Cojocaru} and M.~R. {Murty}.
\newblock {\em {An introduction to sieve methods and their applications}},
  volume~66 of {\em {London Math. Soc. Student Texts}}.
\newblock Cambridge: Cambridge University Press, 2006.

\bibitem[HR]{HelfRadz}
H.~Helfgott and Radziwi{\l}{\l}.
\newblock Expansion, divisibility and parity.
\newblock Submitted. Available at {\url{{https://arxiv.org/abs/2103.06853}}}.

\bibitem[HU]{HelfUbis}
H.~Helfgott and A.~Ubis.
\newblock Primos, paridad y an\'alisis.
\newblock To appear in the proceedings of the AGRA III school. Available at
  {\url{https://arxiv.org/abs/1812.08707}}.

\bibitem[MR16]{MR3488742}
K.~Matom\"aki and M.~Radziwi\l\l.
\newblock Multiplicative functions in short intervals.
\newblock {\em Ann. of Math. (2)}, 183(3):1015--1056, 2016.

\bibitem[MRT15]{MR3435814}
K.~Matom\"aki, M.~Radziwi\l\l, and T.~Tao.
\newblock An averaged form of {C}howla's conjecture.
\newblock {\em Algebra Number Theory}, 9(9):2167--2196, 2015.

\bibitem[Tao16]{MR3569059}
T.~Tao.
\newblock The logarithmically averaged {C}howla and {E}lliott conjectures for
  two-point correlations.
\newblock {\em Forum Math. Pi}, 4:e8, 36, 2016.

\bibitem[TT19]{zbMATH07141311}
T.~{Tao} and J.~{Ter\"av\"ainen}.
\newblock {The structure of correlations of multiplicative functions at almost
  all scales, with applications to the Chowla and Elliott conjectures.}
\newblock {\em {Algebra Number Theory}}, 13(9):2103--2150, 2019.

\end{thebibliography}

\end{document}